\documentclass[a4paper,10pt,onecolumn]{article}
\usepackage{graphicx}
\usepackage[textwidth=16cm]{geometry}
\usepackage{amssymb,amsmath,amsthm} 
\usepackage{microtype}
\usepackage{enumitem} 
\usepackage{breakcites}
\usepackage{xcolor}
\usepackage{bm}
\usepackage{bbm} 
\usepackage{subcaption}
\usepackage{algorithm}
\usepackage[noend]{algpseudocode}
\usepackage{url}
\usepackage{authblk}
\usepackage[hidelinks]{hyperref}  
 
\algnewcommand\algorithmicinput{\textbf{Initialization:}}
\algnewcommand\init{\item[\algorithmicinput]}
\algnewcommand\algorithmicevol{\textbf{Evolution:}}
\algnewcommand\evol{\item[\algorithmicevol]}
\algnewcommand\algorithmicawake{\textsf{\textit{{AWAKE}}}}
\algnewcommand\awake{\item[\algorithmicawake]}
\algnewcommand\algorithmicidle{\textsf{\textit{{IDLE}}}}
\algnewcommand\idle{\item[\algorithmicidle]}

\newcommand{\until}[1]{\{1,\ldots,#1\}}

\newcommand{\DD}{\Pi}
\newcommand{\EE}{\mathcal{E}}
\newcommand{\GG}{\mathcal{G}}

\newcommand{\NN}{\mathcal{N}}

\newcommand{\mS}{\mathcal{S}}

\newcommand{\UU}{{D}}
\newcommand{\VV}{\mathcal{V}} 
\newcommand{\WW}{{W}}

\newcommand{\E}{\mathbb{E}}
\newcommand{\gs}{\textsl{g}}
\newcommand{\ff}{f_\textsl{best}}
\newcommand{\lxi}{x_j^t{\mid}_i}
\newcommand{\bG}{\bar{G}}

\newcommand{\m}{\mathop{\textrm{minimize}}}

\newcommand{\R}{\mathbb{R}}

\newcommand{\x}{\bm{x}}

\newcommand{\prox}{\text{prox}}

\newcommand{\diag}[1]{\text{diag}\{#1\}} 
\newcommand{\NNio}{\NN_{i,out}}
\newcommand{\NNii}{\NN_{i,in}}
\newcommand{\tNNii}{\NN_{i,in}}

\newcommand{\DBS}{Distributed Block Subgradient Method }
\newcommand{\DBP}{Distributed Block Proximal Method }
\newcommand{\DBPnospace}{Distributed Block Proximal Method}

\newcommand{\zl}{z_{\ell}}
\newcommand{\bzl}{\bar{z}_{\ell}}
\newcommand{\bx}{\bar{x}}
\newcommand{\tWl}{{W}_{\ell}}

\newcommand{\Pl}{\Phi_{W_\ell}}
\newcommand{\bS}{\bar{S}}
\newcommand{\bR}{\bar{R}}
\newcommand{\bQ}{\bar{Q}}

\makeatletter
\newcommand{\pushright}[1]{\ifmeasuring@#1\else\omit\hfill$\displaystyle#1$\fi\ignorespaces}
\newcommand{\pushleft}[1]{\ifmeasuring@#1\else\omit$\displaystyle#1$\hfill\fi\ignorespaces}
\makeatother

\newcommand\oprocendsymbol{\hbox{$\square$}}
\newcommand\oprocend{\relax\ifmmode\else\unskip\hfill\fi\oprocendsymbol}

\makeatletter
\def\blfootnote{\xdef\@thefnmark{}\@footnotetext}
\makeatother

\theoremstyle{plain}
\newtheorem{theorem}{Theorem}

\newtheorem{corollary}{Corollary}
\newtheorem{lemma}{Lemma}

\theoremstyle{definition}
\newtheorem{assumption}{Assumption}

\theoremstyle{remark}
\newtheorem{remark}{Remark}

\title{Randomized Block Proximal Methods \\for Distributed Stochastic Big-Data Optimization}
\author{Francesco Farina\thanks{F. Farina is in the Artificial Intelligence and Machine Learning group at GSK. However, this work was carried out while the author was at the Department of Electrical, Electronic and Information Engineering ``G. Marconi'', Universit{\`a} di Bologna, Bologna, Italy. email: \texttt{francesco.x.farina@gsk.com}} ,
Giuseppe Notarstefano\thanks{G. Notarstefano is with the Department of Electrical, Electronic and Information Engineering ``G. Marconi'', Universit{\`a} di Bologna, Bologna, Italy. email: \texttt{giuseppe.notarstefano@unibo.it}}}

\date{}

\begin{document}
\maketitle

\begin{abstract}
\blfootnote{A preliminary version of this work has appeared in the Proceedings of the 58-th Control and Decision Conference (CDC 2019)~\cite{farina2019subgradient}. The current article consider a more general problem set-up, namely a constrained stochastic optimization one. Also the proposed algorithm is more general since local updates are based on generic proximal mappings, agents can be awake or idle at each iteration, blocks can be drawn according to locally defined (possibly non uniform) probability distributions and local stepsize sequences can be employied. Furthermore, the convergence analysis is also carried out under the assumption of constant stepsizes and an explicit convergence rate is provided. Finally, all the complete theoretical proofs are reported.}
\blfootnote{This result is part of a project that has received funding from the European Research Council (ERC) under the European Union's Horizon 2020 research and innovation programme (grant agreement No 638992 - OPT4SMART).}
\blfootnote{\textcopyright  2020 IEEE.  Personal use of this material is permitted.  Permission from IEEE must be obtained for all other uses, in any current or future media, including reprinting/republishing this material for advertising or promotional purposes, creating new collective works, for resale or redistribution to servers or lists, or reuse of any copyrighted component of this work in other works.}
\blfootnote{Digital Object Identifier \url{10.1109/TAC.2020.3027647}}

  In this paper we introduce a class of novel distributed algorithms for solving stochastic big-data convex optimization problems over directed graphs. In the addressed set-up, the dimension of the decision
  variable can be extremely high and the objective function can be nonsmooth.
  The general algorithm consists of two main steps: a consensus step and an
  update on a single block of the optimization variable, which is then broadcast
  to neighbors. Three special instances of the proposed method, involving
  particular problem structures, are then presented.
  In the general case, the convergence of a dynamic consensus algorithm over
  random row stochastic matrices is shown. Then, the convergence of the proposed
  algorithm to the optimal cost is proven in expected value. Exact convergence
  is achieved when using diminishing (local) stepsizes, while approximate convergence is attained when constant
  stepsizes are employed. The convergence
  rate is shown to be sublinear and an explicit rate is provided in the case of constant stepsizes.
  Finally, the algorithm is tested on a distributed classification problem, first on synthetic data and, then, on a real, high-dimensional, text dataset.
\end{abstract}


\section{Introduction} 
Recent years have witnessed a steadily growing interest in distributed learning and control over networks consisting of multiple smart agents.
Several problems arising in this scenario can be formulated as distributed optimization problems which need to be solved by networks of agents.
In this paper, we focus on the following stochastic \emph{big-data} convex optimization problem, which is to be solved over a network of $N$ interconnected agents,
\begin{equation*} 
    \begin{aligned}
        & \m_{x\in X}
        & & \sum_{i=1}^N \E[h_i(x;\xi_i)],
    \end{aligned}
\end{equation*}
where $X\subseteq \R^n$ is a convex set, $\xi_i\in\R$ is a random variable and the functions $h_i:\R^n\to\R$ are continuous, convex and possibly non smooth.
The optimization variable $x$ is extremely high dimensional and with block
  structure, i.e., $n=\sum_{\ell=1}^B n_\ell$ with $n_\ell$ being the dimension
  of the $\ell$-th block and $B\gg 1$ the number of blocks. 
Regarding the role of stochastic functions in the considered set-up, it is
    worth stressing that they allow agents to deal with various type of problems.
     Among the
others, the case of learning problems involving massive datasets is of
particular interest. In this case, the local objective function typically has the form
$f_i(x)=\frac{1}{m_i}\sum_{r=1}^{m_i} h_i(x, \xi_i^r)$, where $\xi_i^r$,
$r=1,\dots,m_i$, 
are samples uniformly drawn from a certain dataset consisting of $m_i$
elements. When $m_i$ is very large it could be computationally infeasible to
compute a subgradient of the entire $f_i$. On the other side, given $\xi_i^r$,
computing a subgradient of $h_i(x, \xi_i^r)$ is much simpler. Problems of this type are
often referred to as sample average approximation
problems~\cite{kleywegt2002sample}. Other relevant classes of problems include
those of dynamic, or online, optimization problems in which samples generating
functions $h_i$ are processed as they become available~\cite{xiao2010dual,tsianos2013consensus} and settings in which only noisy
subgradients of the objective functions are
available~\cite{ram2010distributed}.

Applying classical distributed algorithms to big-data problems may be infeasible due, e.g., to limitations in the communication bandwidth. In fact, they would require agents to communicate a prohibitive amount of data due to the high dimension of the decision variable.
This calls for tailored distributed algorithms for big-data optimization problems in which only few blocks of the entire (local) solution estimate are sent to neighbors. 
Thus, the literature relevant to this paper can be divided in three main
(partially overlapping) categories: stochastic optimization methods, block
coordinate algorithms and primal distributed algorithms.

\emph{Stochastic optimization algorithms}: To the best of our knowledge, the
first work dealing with stochastic problems has
been~\cite{robbins1951stochastic}. Since this seminal work, there has been a
steady increase in the interest for this type of problems, and algorithms for
solving them (see, e.g.,~\cite{kushner2003stochastic} and references
therein). Among the others, stochastic approximation approaches were presented
in~\cite{nemirovski2009robust,ghadimi2012optimal} and stochastic mirror descent
algorithms have been studied
in~\cite{nedic2014stochastic,dang2015stochastic}. Stochastic gradient
descent algorithms are particularly appealing in learning problems (see,
e.g.,~\cite{bottou2010large}) in which extremely large datasets are involved,
since they allow for batch processing of the data.

\emph{Block coordinate algorithms}: Centralized block coordinate methods have
a long history (see, e.g.,~\cite{beck2013convergence} for a survey). They were
firstly designed for solving smooth problems, but, in the last years, an
increasing number of results have been provided to deal with nonsmooth objective
functions. Two main rules for selecting the block to be updated have been
studied: cyclic (or almost cyclic; see, e.g.,~\cite{xu2017globally}) or
random. In the last case, randomized block coordinate algorithms have been proposed~\cite{richtarik2014iteration,zhao2014accelerated,lu2015complexity,chen2016accelerated,zhang2016accelerated}. Particularly
  relevant for this paper is the work in~\cite{dang2015stochastic}, in which a
  stochastic block mirror descent method with random block updates is proposed.
Parallel block coordinate methods are also a well established strand of
optimization literature, see, e.g.,~\cite{wright2015coordinate}. The work
in~\cite{nesterov2012efficiency} applies to smooth convex functions, while the
ones in~\cite{facchinei2015parallel,richtarik2016parallel,necoara2016parallel} face up composite
optimization problems.
A unified framework for nonsmooth optimization using
block algorithms has been studied in~\cite{razaviyayn2013unified} for centralized and parallel set-ups.

\emph{Distributed algorithms}:
Many distributed optimization algorithms have been proposed in recent
years. In~\cite{tsitsiklis1986distributed} a distributed gradient descent
algorithm was firstly introduced, which is capable to deal with both
deterministic and stochastic convex optimization problems.  When the problems to
be solved involve nonsmooth objective functions, subgradient-based algorithms
have been designed. First examples of such algorithms appeared
in~\cite{nedic2009distributed,johansson2009randomized,nedic2010constrained},
while recent advances involve more sophisticated protocols, to deal with
  directed
  communication~\cite{nedic2015distributed,xi2017distributed,liu2017convergence,wang2018distributed,li2019distributed}.
Many distributed algorithms involving proximal operations have also been
  proposed (see, e.g.,\cite{parikh2014proximal} for a survey on proximal
  algorithms). Among the others, a proximal gradient method was developed
  in~\cite{chen2012fast} to deal with unconstrained problems, while
  in~\cite{duchi2012dual,margellos2018distributed} proximal algorithms have been
  presented to deal with constrained optimization.
The stochastic setting has also been
  treated~\cite{ram2010distributed,agarwal2011distributed,srivastava2011distributed,rabbat2015multi,nedic2016stochastic,lan2017communication,li2018stochastic}. In
  particular, a stochastic subgradient projection algorithm appeared
  in~\cite{ram2010distributed}, while a stochastic distributed mirror descent
  was proposed in~\cite{li2018stochastic}.
Distributed algorithms over random networks are also relevant to this
paper. In~\cite{liu2011consensus}, consensus protocols were studied using random
row-stochastic matrices, while in~\cite{lobel2011distributed} a distributed
subgradient method over random networks with underlying doubly stochastic
matrices has been proposed.
%
Distributed algorithms dealing with block communication have started to appear
only recently. A block gradient tracking scheme has been presented
in~\cite{notarnicola2018distributed} for nonconvex problems with nonsmooth
regularizers, while~\cite{FARINA2019243} proposes an asynchronous algorithm for
nonconvex optimization based on the method of multipliers, which is
implementable block-wise. 
A randomized block-coordinate algorithm for smooth problems with common cost function and linear
constraints has been presented in~\cite{necoara2013random}.

\vspace{1ex}
In this paper, we introduce the \DBPnospace, which models a class of
distributed proximal algorithms, with block communication, for solving
stochastic big-data convex optimization problems with nonsmooth objective
function. 
The communication network is modeled as a directed graph admitting a
doubly stochastic weight matrix. At each iteration, each node is awake with a
certain probability (and idle otherwise). If awake, it performs a consensus
step, computes a stochastic subgradient of a local objective function, and
performs a proximal-based update (depending on the computed subgradient and on
a local stepsize) on a randomly chosen block only. Then, it exchanges with its
neighbors only the updated block of the decision variable, thus requiring a
small amount of communication bandwidth.
  We also present three special instances of the proposed algorithm. In the first
  one, the proximal mapping is based on the squared 2-norm, thus leading to
  explicit block subgradient steps. In the other two, smooth objective functions
  and separable (possibly nonsmooth) ones are considered. In both these cases
  the computational load at each node in the network can be further reduced with
  respect to the general algorithm.
  We point out that no global parameter is required in the evolution of the
  algorithms. In fact, each node is awake and selects blocks with \emph{locally
    defined} probabilities, and uses \emph{local} stepsizes.
The block-wise updates and the communication of a single block induce nontrivial technical challenges in the algorithm analysis.
On this regard, it is worth noting that, despite the double stochasticity of the
weight matrix, the consensus step on each block turns out to be performed using
a sequence of random row-stochastic matrices.
The analysis for the \DBP is carried out in two parts. First, the convergence
properties of a dynamic block consensus protocol over random graphs are studied,
by building on block-wise, perturbed consensus dynamics with random matrices. A
bound on the expected distance from consensus is provided, which is then
specialized to the cases of constant and diminishing stepsizes
respectively. 
Then, a bound on the
expected distance from the (globally) optimal cost is provided by properly
bounding errors due to the block-wise update and exploiting the probability of
drawing blocks. When constant stepsizes are used, approximate convergence (with
a constant error term) to the optimal cost is proven in expected value, while
asymptotic exact convergence is reached for diminishing stepsizes. 
Finally, we provide an explicit convergence rate for the proposed algorithm when using constant stepsizes. 
The rate is sublinear, even though a linear term is present, which can be predominant in the first iterations. 

\vspace{1ex}

The paper is organized as follows. The problem set-up is introduced in
Section~\ref{sec:problem} along with some preliminary results. In
Section~\ref{sec:algorithm}, the \DBP is presented and three special algorithm
instances are given in Section~\ref{sec:special}. Then, the algorithm is
analyzed in Section~\ref{sec:analysis}. Finally, a numerical example involving a distributed classification problem over a syntetic and a real, high-dimensional, text document datasets is
dispensed in Section~\ref{sec:experiment} and some conclusions are drawn in
Section~\ref{sec:conclusion}.

\section{Set-up and preliminaries}\label{sec:problem}
\subsection{Notation and definitions}
Given a vector $x\in\R^n$, we denote by $x_\ell$ the
$\ell$-th block of $x$, i.e., given a partition of the identity matrix
$I=[U_1,\dots,U_B]$, with $U_\ell\in\R^{n\times n_\ell}$ for all $\ell$ and
$\sum_{\ell=1}^B n_\ell=n$, it holds $x = \sum_{\ell=1}^B U_\ell x_\ell$
and $x_\ell=(U_\ell)^\top x$. Moreover we denote by $\|x\|$ the 2-norm of $x$. Given a vector $x\in\R^n$, with scalar blocks, we define
$$
d(x)\triangleq\max_{1\leq \ell \leq n}x_\ell-\min_{1\leq \ell\leq n}x_\ell.
$$
Given a vector $x_i\in\R^n$, we denote by $x_{i,\ell}$ the $\ell$-th block of
$x_i$. Moreover, given a constant $c$, and an index $t$, we denote by $(c)^t$,
$c$ to the power of $t$, while given a sequence $\{x^t\}_{t\geq 0}$, we denote
by $x^t$ the $t$-th element of the sequence.  Given a matrix $A$, we denote by
$a_{ij}$ (or $[A]_{ij}$) the element of $A$ located at row $i$ and column $j$. Given two matrices $A$
and $B$, we write $A\geq B$ if $a_{ij}\geq b_{ij} $ for all $i$ and $j$. Given
two vectors $a,b\in\R^n$ we
denote %
by $\langle a,b\rangle$ their scalar product. Given a discrete random variable
$r\in\until{R}$, we denote by $P(r=\bar{r})$ the probability of $r$ to be equal
to $\bar{r}$. Given a nonsmooth function $f$, we
denote by $\partial f(x)$ its subdifferential computed at $x$, and by
$\partial_{x_\ell} f(x)$ the subdifferential of $f$ with respect to the
$\ell$-th block of $x$.
 
We say that a directed
graph $\GG=(\VV,\EE)$ contains a spanning tree if for some $v\in\VV$ there
exists a directed path from the vertex $v$ to all other vertices $u\in\VV$. Given a nonnegative matrix $A$ and some $\delta\in(0,1)$, we denote by
$A_\delta$ the matrix whose entries are defined as
$$
[A_\delta]_{ij}=
\begin{cases}
  \delta,&\text{if }A_{ij}\geq\delta,\\
  0,&\text{otherwise}.
\end{cases}
$$
We say that $A$ contains a $\delta$-spanning tree if the graph induced by
$A_\delta$ contains a spanning tree. 

\subsection{Distributed stochastic optimization set-up}
As anticipated in the introduction, we consider the following optimization problem,
\begin{equation}\label{pb:problem}
    \begin{aligned}
        & \m_{x\in X}
        & & \sum_{i=1}^N \E[h_i(x;\xi_i)].
    \end{aligned}
\end{equation}
We recall that $\xi_i$ is a random variable, functions $h_i:\R^n\to\R$ are
continuous, convex and possibly nonsmooth for every $\xi_i$, $X\subseteq \R^n$ and
$n\gg 1$.  We let $f_i(x)=\E[h_i(x;\xi_i)]$ and $f(x)=\sum_{i=1}^N
f_i(x)$. Moreover, $x^\star\in\R^n$ is a solution of problem~\eqref{pb:problem}.
The optimization variable $x\in\R^n$ has a block structure, i.e.,
\[
  x=[x_1^\top,\dots,x_B^\top]^\top.
\]
with $x_\ell\in\R^{n_\ell}$ for all $\ell$ and $\sum_{\ell=1}^B n_\ell = n$.
We make the following assumption on the problem structure
\begin{assumption}[Problem structure]\label{assumption:problem_structure}
    \hspace{1ex}
    \begin{enumerate}[label=(\Alph*)]
    \item\label{assumption:set} The constraint set $X$ has the block structure 
    $$X = X_1\times \dots\times X_B,$$
    where, for $\ell=1,\dots,B$, the set $X_\ell\subseteq\R^{n_\ell}$ is closed and convex, and $\sum_{\ell=1}^B n_\ell = n$.
    \item\label{assumption:unbiased} Let $g_{i}(x;\xi_i)\in\partial h_i(x;\xi_i)$ (resp. $\gs_i(x)\in\partial f_i(x)$) be a subgradient of $h_i(x;\xi_i)$ (resp. $f_i(x)$) computed at $x$. Then, $g_{i}(x;\xi_i)$ is an unbiased estimator of the subgradient of $f_i$, i.e.,
    \begin{equation*}
        \E[g_{i}(x;\xi_i)]=\gs_i(x).
    \end{equation*}
    \item\label{assumption:subgradients} There exist constants $G_{i}\in[0,\infty)$ and $\bG_{i}\in[0,\infty)$
    such that 
    $$
    \E[\|g_i(x;\xi_i)\|]\leq G_{i},\qquad\E[\|g_i(x;\xi_i)\|^2]\leq \bG_{i},
    $$ 
    for all $x$ and $\xi_i$, for all $i\in\until{N}$.\oprocend
    \end{enumerate}
\end{assumption}
Notice that, if $X=\R^n$, Assumption~\ref{assumption:problem_structure}\ref{assumption:set} is clearly satisfied.
Moreover, let us denote by $g_{i,\ell}(x;\xi_i)$
the $\ell$-th block of $g_{i}(x;\xi_i)$ and let $\gs(x)\in\partial f(x)$ be a subgradient of $f$ computed
at $x$. Then, Assumption~\ref{assumption:problem_structure}\ref{assumption:subgradients} implies that $\E[\|g_{i,\ell}(x;\xi_i)\|]\leq G_i$ for all $\ell\in\until{B}$ and $\|\gs_i(x)\|\leq G_i$.
Moreover, let $\bG\triangleq\sum_{i=1}^N \bG_i$ and $G\triangleq\sum_{i=1}^N G_i$. Then, $\|\gs(x)\|\leq G$ and $\|\gs_i(x)\|\leq G$ for all $i$.

Problem~\eqref{pb:problem} is to be solved in a distributed way by a network
of $N$ agents. Each agent in the network is assumed to know only a portion of
the entire problem, namely agent $i$ knows $f_i$ and the constraint set $X$
only.  We make the following assumption on the network structure.
\begin{assumption}[Communication structure]\label{assumption:communication}
    \hspace{1ex}
    \begin{enumerate}[label=(\Alph*)]
        \item\label{assumption:graph} The network is modeled through a weighted \emph{strongly connected} directed graph $\GG=(\VV,\EE, \WW)$ with $\VV=\until{N}$, $\EE\subseteq\VV\times\VV$ and $\WW\in\R^{N\times N}$ being the weighted adjacency matrix. We denote by $\NNio$ the set of out-neighbors of node $i$, i.e.,
        $\NNio\triangleq\{j\mid (i,j)\in\EE\}\cup\{i\}$.  Similarly, the set of in-neighbors of node
        $i$ is defined as $\NNii\triangleq\{j\mid (j,i)\in\EE\}\cup\{i\}$. 
        \item\label{assumption:stochastic} For all $i,j\in\until{N}$, the weights $w_{ij}$ of the weight matrix $\WW$ satisfy
        \begin{enumerate}[label=(\roman*)]
            \item if $i\neq j$, $w_{ij}>0$ if and only if $j\in\NNii$;
            \item there exists a constant $\eta>0$ such that $w_{ii}\geq\eta$ and if $w_{ij}>0$, then $w_{ij}\geq\eta$;
            \item $\sum_{j=1}^N w_{ij}=1$ and $\sum_{i=1}^N w_{ij}=1$.\oprocend
        \end{enumerate} 
    \end{enumerate}
\end{assumption}

A function $\omega_{\ell}$ is associated to the
$\ell$-th block of the optimization variable for all $\ell$. Let the function
$\omega_{\ell}:X_\ell\to\R$, be continuously differentiable and
$\sigma_{\ell}$-strongly convex. Functions $\omega_{\ell}$ are sometimes
referred to as distance generating functions. Then, we define the proximal
function, also called \emph{Bregman's divergence}, associated to $\omega_{\ell}$ as
\begin{equation*}
    \nu_{\ell}(a,b)=\omega_{\ell}(b)-\omega_{\ell}(a)-\langle \nabla \omega_{\ell} (a), b-a\rangle,
\end{equation*}
for all $a,b\in X_\ell$. The following assumption is made on the functions $\nu_{\ell}$.
\begin{assumption}[Bregman's divergence separate convexity]\label{assumption:separate}
    For all $\ell\in\until{B}$, the function $\nu_{\ell}$ satisfies
    \begin{equation}\label{eq:separable}
        \nu_{\ell}\left(\sum_{j=1}^N \theta_j a_j,b\right)\leq \sum_{j=1}^N \theta_j \nu_{\ell}(a_j,b), \quad \forall a_1,\dots,a_N,b\in X_\ell,
    \end{equation}
    where $\sum_{j=1}^N \theta_j=1$ and $\theta_j\geq 0$ for all $j$.\oprocend
\end{assumption}
Notice that the above assumption is satisfied by many functions (such as the
quadratic function, the Boltzmann-Shannon entropy and the exponential function)
and conditions on the functions $\omega_{\ell}$
guaranteeing~\eqref{eq:separable} can be provided
(see~\cite{bauschke2001joint}).  
Finally, given $a\in X_\ell$, $b\in\R^{n_\ell}$
and $c\in\R$, the proximal mapping associated to $\nu_{\ell}$ is defined as
\begin{equation}\label{eq:prox}
    \prox_{\ell} (a,b,c)=\arg\min_{u\in X_\ell}\left\lbrace\langle b, u \rangle+\frac{1}{c} \nu_{\ell}(a, u)\right\rbrace.
\end{equation}

\subsection{Preliminary results}
Consider a stochastic, discrete-time dynamical system evolving according to
\begin{equation}\label{eq:system}
    x^{t+1}=A^tx^t,\quad\forall t,
\end{equation}
where $\{A^t\}_{t\geq 0}$ is a sequence of random $n\times n$ row-stochastic matrices.
Let $(\Omega,\mathcal{F},P)$ be a probability space. We assume that the sequence
$\{A^t,\mS^t\}_{t\geq 0}$ forms an \emph{adapted process}, i.e., $\{A^t\}_{t\geq 0}$ is a
stochastic process defined on $(\Omega,\mathcal{F},P)$, $\{\mS^t\}_{t\geq 0}$ is a
filtration (i.e., $\mS^{t}\subseteq\mS^{t+1}$ and $\mS^{t}\subseteq\mathcal{F}$ for
all $t$) and $A^t$ is measurable with respect to $\mS^t$.  Given a sequence of
matrices $\{A^t\}_{t\geq 0}$, let us define the transition matrix from iteration $s$
to iteration $t$ as
\begin{equation*}
    \Phi_{A}^{t,s}\triangleq 
    \begin{cases}
        A^t A^{t-1}\dots A^s,&\text{if }t>s,\\
        A^t,&\text{if }t=s.
    \end{cases}
  \end{equation*}
  Then, the following result, adapted from \cite[Theorem~3.1]{liu2011consensus},
  holds true for system~\eqref{eq:system}.
\begin{lemma}[{\cite[Theorem~3.1]{liu2011consensus}}]\label{lemma:consensus_stochastic}
    Consider system~\eqref{eq:system}. If there exist $h>0$, $\delta>0$ such that $\E[\sum_{t=mh+1}^{(m+1)h}A^t\mid\mS^{mh}]$ contains a $\delta$-spanning tree for each $m$, and $A^t\geq\delta I$ for each $t$, then, for any given initial distribution of $x^0$ with $\E[\|x^0\|^p]<\infty$ (which is independent of $\{A^t\}_{t\geq 0}$), and any $p>0$, it holds
    \begin{align*}
        \E[d(x^t)^p]&=\E[d(\Phi_{A}^{t,0}x^0)^p] \nonumber\\
        &\leq(\mu)^t \E[d(x^0)^p]\leq M(\mu)^t\E[\|x^0\|^p],
    \end{align*}
    where $M\in(0,\infty)$ and $\mu\in(0,1)$.\oprocend
\end{lemma}
Finally, the following three results will be useful in the rest of the paper.
\begin{lemma}\label{lemma:series}
    Given a scalar $\beta\neq 1$, it holds that
    \begin{enumerate}[label=(\roman*)]
        \item\label{item:1} for any $t\geq r\geq 0$, $\sum_{s=r}^t(\beta)^s = \frac{(\beta)^r-(\beta)^{t+1}}{1-\beta}$
        \item\label{item:2} for any $t\geq 0$, $\sum_{s=0}^t\sum_{\tau=0}^{t-s}(\beta)^\tau = \frac{t+1-\beta(t+2)+(\beta)^{t+2}}{(1-\beta)^2}$\oprocend
    \end{enumerate}
\end{lemma}
\begin{lemma}[{\cite[Lemma~3.1]{ram2010distributed}}]\label{lemma:series_nedic}
    Let $\{\gamma^t\}_{t\geq 0}$ be a scalar sequence. 
    \begin{enumerate}[label=(\roman*)]
        \item If $\lim_{t\to\infty}\gamma^t=\gamma$ and $\beta\in(0,1)$ then $\lim_{t\to\infty}\sum_{s=0}^t(\beta)^{t-s}\gamma^s=\frac{\gamma}{1-\beta}$.
        \item If $\gamma^t\geq 0$ $\forall t$, $\sum_{t=0}^\infty \gamma^t<\infty$ and $\beta\in(0,1)$, then $\sum_{t=0}^\infty\left(\sum_{s=0}^t(\beta)^{t-s}\gamma^s \right)<\infty$.\oprocend
    \end{enumerate}
\end{lemma}
\begin{lemma}[Tower property of conditional expectation]\label{lemma:expectation}
  Let $X$ be a random variable defined on a probability space
  $(\Omega,\mathcal{F},P)$. Let $\mathcal{Z}\subseteq \mathcal{Y}\subseteq\mathcal{F}$. Then, $
    \E[\E[X\mid \mathcal{Y},\mathcal{Z}]\mid \mathcal{Z}]=\E[X\mid \mathcal{Z}].$\oprocend
\end{lemma}

\section{Distributed Block Proximal Method}\label{sec:algorithm}
The \DBP for solving problem~\eqref{pb:problem} in a distributed way is now
introduced. %
The algorithm works as follows. Each
agent $i$ maintains a local solution estimate $x_i^t$ and a local copy of the
estimates of its in-neighbors. Let us denote by $\lxi$ the copy of the solution
estimate of agent $j$ at agent
$i$. 
At the beginning, each node initializes its state with a random
(bounded) initial condition $x_i^0$ which is then shared with its
neighbors.
At each iteration each agent $i$ is awake with probability $p_{i,on}\in (0,1]$
and idle with probability $1-p_{i,on}$.
Thus, the proposed algorithm models a particular type of asynchrony
in which the communication graph is fixed and agents can communicate or not with
their neighbors with a certain probability. 
If agent $i$ is awake, it picks randomly a block $\ell_i^t\in\until{B}$, some
$\xi_i^t$, and performs two updates:
\begin{enumerate}[label=(\roman*)]
    \item it computes a weighted average of its in-neighbors' estimates $\lxi$, $j\in\tNNii$;
    \item it computes
      $x_i^{t+1}$ by updating the $\ell_i^t$-th block of $x_i^{t}$ through a proximal mapping step and leaving the other blocks unchanged.
\end{enumerate}
Then, it broadcasts $x_{i,\ell_i^t}^{t+1}$ to its out-neighbors. We model the
status (awake or idle) of each node $i$ at each iteration $t$ through a random
variable $s_i^t\in\{0,1\}$ which is $1$ (corresponding to being awake) with
probability $p_{i,on}$ and $0$ with probability $1-p_{i,on}$.  A pseudocode of
the method is reported in Algorithm~\ref{alg:DBS}.

\begin{algorithm}
    \begin{algorithmic}
        \init $x_i^0$
		\evol for $t=0,1,\dots$
         
        \State \textsc{Update} for all $j\in\NNii$
        \begin{equation}\label{eq:xl_update}
            x_{j,\ell}^t{\mid}_i = 
            \begin{cases}
                x_{j,\ell}^t, &\text{if }\ell=\ell_j^{t-1} \; \text{and} \; s_j^{t-1}=1\\
                x_{j,\ell}^{t-1}{\mid}_i, &\text{otherwise}
            \end{cases}
        \end{equation}
        \If{$s_i^t=1$ }
        \State\textsc{Pick} $\ell_i^{t}\in\until{B}$ with $P(\ell_i^t=\ell)=p_{i,\ell}>0$, $\forall\ell$
        \State\textsc{Compute} 
		\begin{equation}\label{eq:y_update_l}
			y_i^{t} = \sum_{j\in\tNNii} w_{ij} \lxi
		\end{equation}
        \State\textsc{Update} 
        \begin{equation}\label{eq:x_update_l}
            x_{i,\ell}^{t+1} = 
            \begin{cases}
                \prox_{\ell}(y_{i,\ell}^t,g_{i,\ell}(y_i^t;\xi_i^t),\alpha_i^t),&\text{if } \ell=\ell_i^t\\
                x_{i,\ell}^{t},&\text{otherwise}
            \end{cases}
        \end{equation}
        \State\textsc{Broadcast} $x_{i,\ell_i^t}^{t+1}$ to $j\in\NNio$
        \Else{ $x_i^{t+1}=x_i^t$}
        \EndIf
		
	\end{algorithmic}
	\caption{\DBP}\label{alg:DBS}
\end{algorithm} 

Notice that all the quantities involved in the above algorithm are local for
each node. In fact, each node has locally defined probabilities (both of
awakening and block drawing) and local stepsizes.

Moreover, it is worth noting that, despite node $i$ receives from each
$j\in\NN_i^{in}$ only the block $x_{j,\ell_j^{t-1}}^{t}$, the consensus
step~\eqref{eq:y_update_l} is in fact performed by using the entire
$x_j^{t}$. Indeed, the other blocks have not changed since the last time they
have been received. This is formalized in the next result.
\begin{lemma}\label{lemma:equivalent}
    Let Assumption~\ref{assumption:communication} hold. Then $x_{j}^t{\mid}_i=x_j^t$ for all $t$. Moreover, Algorithm~\ref{alg:DBS} can be compactly rewritten as follows. For all $i\in\until{N}$ and all $t$, if $s_i^t=1$,
    \begin{align}
        y_i^{t} &= \sum_{j=1}^N w_{ij} x_j^t,\label{eq:y_update}\\
        x_{i,\ell}^{t+1} &= 
        \begin{cases}
            \textup{prox}_{\ell}(y_{i,\ell}^t,g_{i,\ell}(y_i^t;\xi_i^t),\alpha_i^t),&\text{if } \ell=\ell_i^t,\\
            x_{i,\ell}^{t},&\text{otherwise},\label{eq:x_update}
        \end{cases}
    \end{align}
    else, $x_i^{t+1}=x_i^t$.
\end{lemma}
\begin{proof}
  The fact that $x_{j}^t{\mid}_i=x_j^t$ for all $i$ and all $t$ follows
  immediately from the evolution of the algorithm. In fact, the received block
  $x_{j,\ell_j^{t-1}}^t{\mid}_i$ is the only block that node $j$ has modified in
  the last iteration, while the others have remained unchanged. Hence, since the
  graph $\GG$ is fixed, it is clear that $x_{j}^t{\mid}_i=x_j^t$ for all $i$ and
  all $t$. The reformulation of Algorithm~\ref{alg:DBS}
  as~\eqref{eq:y_update}-\eqref{eq:x_update} is then immediate from Assumption~\ref{assumption:communication}\ref{assumption:stochastic}.
\end{proof}
In virtue of the previous result, in order to lighten the notation in the subsequent analysis, we will use~\eqref{eq:y_update}-\eqref{eq:x_update} in place of Algorithm~\ref{alg:DBS}, by making the block communication implicit. 

As for the block-wise proximal update~\eqref{eq:x_update}, the $\ell_i^t$-th
block of a \emph{whole} stochastic subgradient computed at $y_i^{t}$ is used.
Unfortunately, computing a subgradient with respect to the $\ell_i^t$-th
component only is, in general, \emph{not} equivalent to picking the
$\ell_i^t$-th block of a \emph{whole} subgradient $g_i(y_i^t;\xi_i^t)$. In fact, in general it holds that, picking $g_1\in\partial_{y_{i,1}} h_i(y_i^t;\xi_i^t)$,\dots, $g_B\in\partial_{y_{i,B}} h_i(y_i^t;\xi_i^t)$ does not imply $[g_i^\top,\dots,g_B^\top]^\top\in \partial h_i(y_i^t;\xi_i^t)$.
This will turn out to be extremely important in the subsequent analysis.
If functions $f_i$ are separable on the blocks, then, only the subgradient with
respect to the $\ell_i^t$-th component can be computed.
Similarly, if the
functions $f_i$ are smooth, the $\ell_i^t$-th block of the gradient can be
directly computed as the gradient with respect to that block.
In these cases, the computational load at each node can be further reduced, as it will be shown in Section~\ref{sec:special}.

The last key feature of the \DBP involves the consensus
step~\eqref{eq:y_update}. Let $\zl^t$ be the vector stacking the $\ell$-th
component of all the $x_i^t$, i.e.,
$\zl^t\triangleq[(x_{1,\ell}^t)^\top, \dots, (x_{N,\ell}^t)^\top]^\top$. Also,
let $\UU_{\ell}^t$ be a diagonal matrix in which the $i$-th element of the
diagonal is set to $1$ if $s_i^t=1$ and $\ell_i^t=\ell$, and it is set to $0$
otherwise, i.e.,
\begin{equation*}
    [\UU_{\ell}^t]_{ij} = \begin{cases}
        1,&\text{if }i=j\text{, }\ell=\ell_i^t \text{ and }s_i^t=1,\\
        0,&\text{otherwise}.
    \end{cases}
\end{equation*}
Finally, let $\UU_{-\ell}^t=I-\UU_{\ell}^t$.
Now, consider a consensus protocol associated to the \DBPnospace, i.e.,
\begin{align*}
    y_i^{t} &= \sum_{j=1}^N w_{ij} x_j^t,\\
    x_{i,\ell}^{t+1} &= 
    \begin{cases}
        y_{i,\ell}^t,&\text{if } \ell=\ell_i^t \text{ and } s_i^t=1,\\
        x_{i,\ell}^{t},&\text{otherwise}.
    \end{cases}
\end{align*}
This system can be rewritten in terms
of $\zl$ as
\begin{equation*}
    \zl^{t+1}=\tWl^{t}\zl^{t},
\end{equation*}
where $\tWl^{t}\triangleq\UU_{-\ell}^{t} + \UU_{\ell}^{t} \WW$.
It can be easily verified that, for all $\ell$ and $t$, the matrix $\tWl^t$ is
row-stochastic but not doubly stochastic anymore (unless all nodes select the
same block $\ell$ at some iteration $t$).

\begin{remark}
It is worth noting that the proposed algorithm, besides being easy to implement, considers challenges that cannot be addressed by other block-wise distributed algorithms~\cite{notarnicola2018distributed,FARINA2019243,necoara2013random}. In particular, none of those works deals with stochastic problems. Moreover, in~\cite{notarnicola2018distributed} composite objective functions with non-smooth components are considered but the non-smooth part must be common to all the agents. In~\cite{FARINA2019243,necoara2013random} at least differentiability of the objective is required. Finally, in our algorithm, all the algorithm parameters are local.%
\end{remark}

\section{Special instances} 
\label{sec:special}

In this section, three special cases of the \DBP are presented. The first one is
obtained by choosing the squared 2-norm as distance generating function, while the other two result from smooth and separable objective functions respectively.

\subsection{Distributed Block Subgradient Method}
By using $\omega_{\ell}(x)=\frac{1}{2}\|x\|^2$ for all $\ell$, and
assuming $X=\R^n$, the proximal mapping~\eqref{eq:prox} has an explicit
analytical solution and the update step~\eqref{eq:x_update_l} becomes
\begin{equation}\label{eq:x_update_l_subgradient}
    x_{i,\ell}^{t+1} = 
    \begin{cases}
        y_{i,\ell}^t - \alpha_i^t g_{i,\ell}(y_i^t;\xi_i^t), &\text{if } \ell=\ell_i^t\\
        x_{i,\ell}^{t},&\text{otherwise}.
    \end{cases}
\end{equation}
Notice that, the proximal step becomes a subgradient step on a single block of
the optimization variable. Thus, we call \DBS the resulting algorithm, i.e., the
one obtained by replacing~\eqref{eq:x_update_l}
with~\eqref{eq:x_update_l_subgradient} in Algorithm~\ref{alg:DBS}. Notice
  that, in this case, it holds that, for all $\ell$, the strong
  convexity parameter is $\sigma_{\ell}=1$, thus resulting in special bounds
  in the subsequent algorithm analysis.

\subsection{Smooth functions}
    The update of the solution estimate in~\eqref{eq:x_update_l} requires, in general, for node $i$ at iteration $t$, the computation of an entire stochastic subgradient at the point $y_i^t$. However, only the $\ell_i^t$-th block of the computed subgradient is used in the update step. When a function $h_i$ is smooth, however, the $\ell_i^t$-th block of its gradient can be directly computed as the gradient of $h_i$ with respect to the $\ell_i^t$-th block of the optimization variable and~\eqref{eq:x_update_l} can be replaced by
    \begin{equation}\label{eq:x_update_l_smooth}
        x_{i,\ell}^{t+1} = 
        \begin{cases}
            \prox_{\ell}(y_{i,\ell}^t,\nabla_{\ell}h_{i}(y_i^t;\xi_i^t),\alpha_i^t), &\text{if } \ell=\ell_i^t,\\
            x_{i,\ell}^{t},&\text{otherwise},
        \end{cases}
    \end{equation}
    where $\nabla_{\ell}h_{i}$ denotes the (partial) gradient of $h_i$ with respect to the $\ell$-th block of the optimization variable. Thus, when smooth functions are involved in the problem, the computational load can be reduced by avoiding the computation of the entire (sub)gradient. 

\subsection{Separable functions}
    When functions $h_i(x;\xi_i)$ are separable, i.e., 
    $$
        h_i(x;\xi_i) = \sum_{\ell=1}^B \hat{h}_{i,\ell}(x_\ell,\xi_i),
    $$
    the \DBP can be further simplified, allowing for an extra reduction of the computational load at each iteration at a given node. In fact,
    it holds that
    $\partial h_i(y_i^t;\xi_i^t)= \partial_{y_{i,1}}
    h_i(y_i^t;\xi_i^t)\times\dots \partial_{y_{i,B}} h_i(y_i^t;\xi_i^t)$, and
    hence
    $ g_i(y_i^t;\xi_i^t)=\sum_{\ell=1}^B
    \hat{g}_{i,\ell}(y_{i,\ell}^{t},\xi_i^t), $ where
    $\hat{g}_{i,\ell}\in \partial \hat{h}_{i,\ell}$ is a subgradient of
    $\hat{h}_{i,\ell}$. This implies that
    $g_{i,\ell}(y_{i}^{t},\xi_i^t)=\hat{g}_{i,\ell}(y_{i,\ell}^{t},\xi_i^t)$
    and, thus, only the $\ell_i^t$-th block of $y_i^t$ is needed in order to
    compute $g_{i,\ell}(y_{i}^{t},\xi_i^t)$ and hence $x_i^{t+1}$. Thus the \DBP
    can be simplified by allowing nodes with a separable function to reduce
    their computational load. In particular, assume the cost function of
      node $i$ to be separable. Then, a single block of $y_i^t$ can be updated
      at each iteration and a subgradient can be directly computed for the
      corresponding block, without computing an entire subgradient. Hence, the
      algorithm can be rewritten, by using the equivalent formulation in
      Lemma~\ref{lemma:equivalent}, as follows. If $s_i^t=1$,
    \begin{align} 
        y_{i,\ell}^{t} &=   
        \begin{cases}
        \sum_{j=1}^N w_{ij} x_{j,\ell}^t,&\text{if } \ell=\ell_i^t,\\
        y_{i,\ell}^{t-1},&\text{else},\label{eq:y_update_sep}
        \end{cases}\\
        x_{i,\ell}^{t+1} &= 
        \begin{cases}
            \textup{prox}_{\ell}(y_{i,\ell}^t,\hat{g}_{i,\ell}(y_{i,\ell}^{t},\xi_i^t),\alpha_i^t),&\text{if } \ell=\ell_i^t,\\
            x_{i,\ell}^{t},&\text{otherwise},\label{eq:x_update_sep}
        \end{cases}
    \end{align}
    else, $x_i^{t+1}=x_i^t$.

\section{Algorithm analysis}\label{sec:analysis}
In this section, the convergence of the \DBP is proven in expected value. The
proof consists of two main parts. In the first one the consensus of the agents'
solution estimates is shown, while in the second one convergence towards the
optimal cost is proven. Both results are given, at first, in a general form and,
then, specialized to the case of constant stepsizes (in which convergence to a
neighborhood is proven) and diminishing stepsizes (in which exact asymptotic
convergence is reached).

Define $a^t\triangleq[\alpha_1^t,\dots,\alpha_N^t]^\top$, $a_M^t\triangleq\max_{i}\alpha_i^t$ and $a_m^t\triangleq\min_{i}\alpha_i^t$.
We summarize in the following two assumptions, the two different choices for the stepsize sequences we consider in the following analysis.
\begin{assumption}[Constant stepsize]\label{assumption:constant}
    The sequences $\{\alpha_i^t\}_{t\geq 0}$ satisfy $\alpha_i^t=\alpha_i>0$ for all $t$ and all $i$. 
\end{assumption}
\begin{assumption}[Diminishing stepsize]\label{assumption:stepsize}
    The sequences $\{\alpha_i^t\}_{t\geq 0}$ satisfy
    \begin{equation*}
        \sum_{t=0}^\infty \alpha_i^t = \infty, \qquad \sum_{t=0}^\infty (\alpha_i^t)^2 < \infty,
    \end{equation*}
    for all $i\in\until{N}$.
    Moreover, $\alpha_i^{t+1}\leq\alpha_i^t$ for all $t$ and all $i\in\until{N}$.   
\end{assumption}
Notice that, under Assumption~\ref{assumption:constant}, $a_M^t\triangleq a_M=\max_i \alpha_i$ and $a_m^t \triangleq a_m=\min_i \alpha_i$ for all $t$, while, under Assumption~\ref{assumption:stepsize} it can be easily verified that 
$$
    \sum_{t=0}^\infty a_M^t = \infty, \qquad \sum_{t=0}^\infty (a_M^t)^2 < \infty, \qquad a_M^{t+1}\leq a_M^t,
$$ 
and
$$
    \sum_{t=0}^\infty a_m^t = \infty, \qquad \sum_{t=0}^\infty (a_m^t)^2 < \infty, \qquad a_m^{t+1}\leq a_m^t.
$$

Define the vector stacking all local solution estimates as $\x(t)\triangleq[(x_1(t))^\top,\dots,(x_N(t))^\top]^\top$, and the average (over the agents) of the local estimates at $t$ as
\begin{equation}\label{eq:x_mean}
    \bar{x}(t)\triangleq\frac{1}{N}\sum_{i=1}^N x_i(t).
\end{equation}
Then, we make the following assumption on the random variables involved in the algorithm.
\begin{assumption}[Random variables]\label{assumption:random_variables}
    \hspace{1ex}
    \begin{enumerate}[label=(\Alph*)]
    \item\label{assumption:iid}For a given $i\in\until{N}$, the random variables $\ell_i^t$ and $s_i^t$ are independent and identically distributed for all $t$.
    \item\label{assumption:indepentent}
    For a given $t$, the random variables $s_i^t$, $\ell_i^t$ and $\xi_i^t$ are independent of each other for all $i\in\until{N}$.
    \item\label{assumption:initial}
    There exist constants $C_i\in[0,\infty)$ such that $\E[\|x_i^0\|]\leq C_i$ for all $i\in\until{N}$ and hence $\E[\|\x^0\|]\leq C=\sum_{i=1}^N C_i$.\oprocend
    \end{enumerate}
\end{assumption}

Before proceeding with the algorithm analysis, let us provide a preliminary instrumental result. Define $q_{i}^t=x_{i,\ell_i^t}^{t+1}-y_{i,\ell_i^t}^t$ and $q^t=[(q_1^t)^\top,\dots,(q_N^t)^\top]^\top$. Then, the following result applies.
\begin{lemma}\label{lemma:prox_bound}
    Let Assumptions~\ref{assumption:problem_structure}\ref{assumption:set} and~\ref{assumption:problem_structure}\ref{assumption:subgradients} hold. Then,
    \begin{equation*}
        \E[\|q_{i}^t\|]\leq \frac{G_i}{\sigma}\alpha_i^t,
    \end{equation*}
    for all $i\in\until{N}$, where $\sigma=\min_{\ell}\sigma_{\ell}$. 
\end{lemma}
\begin{proof}
    The first order necessary optimality condition on~\eqref{eq:x_update} for $\ell=\ell_i^t$ reads
    \begin{equation}\label{eq:ei}
        \langle \alpha_i^t g_{i,\ell_i^t}(y_i^t;\xi_i^t) +\nabla \omega_{\ell_i^t}(x_{i,\ell_i^t}^{t+1})-\nabla \omega_{\ell_i^t}(y_{i,\ell_i^t}^{t}), u- x_{i,\ell_i^t}^{t+1}\rangle\geq 0,
    \end{equation}
    for all $u\in X_{\ell_i^t}$. Notice now that, by definition, $y_{i,\ell_i^t}^{t}\in X_{\ell_i^t}$, since it is a weighted average of points lying in $X_{\ell_i^t}$. Thus, by taking $u=y_{i,\ell_i^t}^{t}$, one obtains
    \begin{align}
        \alpha_i^t\langle g_{i,\ell_i^t}(y_i^t;\xi_i^t), y_{i,\ell_i^t}^{t}-x_{i,\ell_i^t}^{t+1}\rangle
        &\geq\langle \nabla \omega_{\ell_i^t}(y_{i,\ell_i^t}^{t})-\nabla\omega_{\ell_i^t}(x_{i,\ell_i^t}^{t+1}), y_{i,\ell_i^t}^{t}- x_{i,\ell_i^t}^{t+1}\rangle\nonumber\\
        &\geq \sigma_{\ell_i^t}\|y_{i,\ell_i^t}^{t}- x_{i,\ell_i^t}^{t+1}\|^2,
    \end{align}
    where we have used the strong convexity of $\omega_{\ell_i^t}$. By rearranging the terms, one has
    \begin{align}
        \sigma_{\ell_i^t}\|y_{i,\ell_i^t}^{t}- x_{i,\ell_i^t}^{t+1}\|^2&\leq \alpha_i^t\langle g_{i,\ell_i^t}(y_i^t;\xi_i^t), y_{i,\ell_i^t}^{t}-x_{i,\ell_i^t}^{t+1}\rangle\nonumber\\
        &\leq\alpha_i^t\|g_{i,\ell_i^t}(y_i^t;\xi_i^t)\|\|y_{i,\ell_i^t}^{t}-x_{i,\ell_i^t}^{t+1}\|
    \end{align}
    and hence, 
    \begin{equation*}
        \|q_i^t\|\leq \frac{\alpha_i^t}{\sigma_{\ell_i^t}}\|g_{i,\ell_i^t}(y_i^t;\xi_i^t)\|\leq \frac{\alpha_i^t}{\sigma}\|g_{i,\ell_i^t}(y_i^t;\xi_i^t)\| .
    \end{equation*}
    Now, by taking the expected value and using the subgradient boundedness from Assumption~\ref{assumption:problem_structure}\ref{assumption:subgradients}, one gets
    \begin{align*}
        \E[\|q_i^t\|]\leq \frac{\alpha_i^t}{\sigma}\E[\|g_{i,\ell_i^t}(y_i^t;\xi_i^t)\|]\leq \frac{\alpha_i^t G_i}{\sigma},
    \end{align*}
    thus concluding the proof.
\end{proof}

\subsection{Dynamic consensus with random matrices}
In this section we show that the sequences $\{x_i^t\}_{t\geq 0}$ and $\{y_i^t\}_{t\geq 0}$ computed by each agent in the network asymptotically achieve consensus in expected value when using diminishing stepsizes. Moreover, an upper bound on the distance from consensus is provided in the case of constant stepsizes.

Let $\mS^t\triangleq \{\x^\tau\mid \tau\in\{0,\dots,t\}\}$ be the set of estimates generated by the \DBP up to iteration $t$ (which is indeed a filtration). 
Moreover, define the probability of node $i$ to both be awake and pick block $\ell$ at each iteration as 
$$
    \pi_{i,\ell}\triangleq p_{i,on} p_{i,\ell}.
$$

Then, the following lemma provides a bound on the expected distance between $x_i^t$ and the average $\bx^t$ (defined in~\eqref{eq:x_mean}).
\begin{lemma}\label{lemma:xi-bx}
  Let Assumptions~\ref{assumption:problem_structure}\ref{assumption:subgradients},~\ref{assumption:communication},~\ref{assumption:random_variables} hold. Then, there exist constants $M\in(0,\infty)$
  and $\mu_M\in(0,1)$ such that
\begin{align}
    \E [\|x_i^{t}-\bx^{t}\|]& \leq MB\Bigg((\mu_M)^{t-1}C+\frac{G}{\sigma}\sum_{s=0}^{t-2}(\mu_{M})^{t-s-2} a_M^s+\frac{G}{\sigma} a_M^{t-1}\Bigg),\label{eq:xi-bx}
\end{align}
for all $i\in\until{N}$ and all $t\geq 1$.
\end{lemma}
\begin{proof}
  For the sake of presentation, assume that the blocks are scalars, i.e., $B=n$.
  Let us recall that $\zl^t$ defines the vector stacking the $\ell$-th component
  of all the $x_i^t$, i.e., $\zl^t\triangleq[x_{1,\ell}^t, \dots, x_{N,\ell}^t]^\top$,
  while the matrix $\UU_{\ell}^t\in\R^{N\times N}$ is a diagonal matrix in which the $i$-th
  element of the diagonal is set to $1$ if $s_i^t=1$ and $\ell_i^t=\ell$ and it is set to $0$
  otherwise, i.e.,
\begin{equation*}
    [\UU_{\ell}^t]_{ij} = \begin{cases}
        1,&\text{if }i=j\text{, }\ell=\ell_i^t\text{ and }s_i^t=1,\\
        0,&\text{otherwise}.
    \end{cases}
\end{equation*}
Consistently, we let $\UU_{-\ell}^t=I-\UU_{\ell}^t$. %
Notice that, for all $t$,
$\UU_{\ell}^t$ is a random matrix whose diagonal
element $[\UU_{\ell}^t]_{ii}$ is $1$ with probability $\pi_{i,\ell}$ and $0$ with
probability $1-\pi_{i,\ell}$. Define 
$$
\DD_\ell\triangleq\diag{[\pi_{1,\ell},\dots,\pi_{N,\ell}]}.
$$
Then, by using Assumption~\ref{assumption:random_variables}\ref{assumption:iid}, it can be verified that
\begin{equation}\label{eq:Ul}
    \E[\UU_\ell^t\mid \mS^{t}]=\E[\UU_\ell^t]= \DD_\ell
\end{equation}
and, similarly,
\begin{equation}\label{eq:Uml}
    \E[\UU_{-\ell}^t\mid \mS^{t}]=\E[\UU_{-\ell}^t]= I-\DD_\ell
\end{equation}
for all $t$.

Now, Algorithm~\ref{alg:DBS} can be rewritten with respecte to $\zl$ as
\begin{equation}\label{eq:zl}
    \zl^{t+1}=\tWl^t\zl^{t}+e_{\ell}^t, 
\end{equation}
where $\tWl^t\triangleq\UU_{-\ell}^t + \UU_{\ell}^t \WW$ is, by definition, a row-stochastic matrix, and $e_{\ell}^t\triangleq\UU_\ell^t q^t$.
Now, by recursively
applying~\eqref{eq:zl}, it holds that
\begin{equation*}
    \zl^{t+1}=\Pl^{t,0}\zl^0+\sum_{s=0}^{t-1}\Pl^{t,s+1}e_{\ell}^s + e_{\ell}^{t},
\end{equation*}
where $\Pl^{t,s}$ is the transition matrix from iteration $s$ to iteration $t$ associated to the matrices
$\tWl^{\tau}$, $\tau=s,\dots,t$. Moreover, by applying the $d(\cdot)$ operator on both sides (recall that
$d(\zl)=\max_{1\leq i\leq N}z_{\ell,i}-\min_{1\leq i\leq N}z_{\ell,i}$), one has
\begin{equation}\label{eq:d_evol}
    d(\zl^{t+1})\leq d(\Pl^{t,0}\zl^0)+\sum_{s=0}^{t-1}d(\Pl^{t,s+1}e^s) + d(e^t).
\end{equation}
Notice now that, by using~\eqref{eq:Ul} and~\eqref{eq:Uml},
\begin{align*}
    \E\left[ \tWl^t \right]&=\E[\UU_{-\ell}^t]+\E[\UU_{\ell}^t] \WW \nonumber\\
    &= I - \DD_\ell + \DD_\ell\WW
\end{align*}
for all $t$. 
It can be seen that such a matrix is row stochastic and contains a $\delta$-spanning tree with $\delta\geq (\eta\min_{i}\pi_{i,\ell}+\min_{i}(1-\pi_{i,\ell}))>0$ (since from Assumption~\ref{assumption:communication} the matrix $\WW$ contains a spanning tree), where $\eta$ is defined in Assumption~\ref{assumption:communication}\ref{assumption:stochastic}. Moreover, by
Assumption~\ref{assumption:random_variables}\ref{assumption:iid}, $\{\tWl^t\}_{t\geq 0}$ is a sequence of
i.i.d. random matrices with $\tWl^t\geq\eta I$. Hence, from Lemma~\ref{lemma:consensus_stochastic}, by
taking the expectation on both sides of~\eqref{eq:d_evol}, we get
\begin{align*}
    \E[d(\zl^{t+1})]&\leq \E[d(\Pl^{t,0}\zl^0)]+\sum_{s=0}^{t-1}\E[d(\Pl^{t,s+1}e_{\ell}^s)]+\E[d(e_{\ell}^t)]\nonumber\\
    &\leq (\mu_\ell)^{t} \E[d(\zl^0)]+\sum_{s=0}^{t-1}(\mu_{\ell})^{t-s-1}\E[d(e_{\ell}^s)]+\E[d(e_{\ell}^t)]\nonumber\\
    &\leq M \left((\mu_\ell)^{t} \E[\|\zl^0\|]+ \frac{G}{\sigma} \sum_{s=0}^{t-1}(\mu_{\ell})^{t-s-1} a_M^s+\frac{G}{\sigma}  a_M^t\right).
\end{align*}
where we used the fact that, from Lemma~\ref{lemma:prox_bound}, 
$$\E[\|e_{\ell}^t\|]\leq \E[\|q^t\|]\leq \sum_{i=1}^N\E[\|q_i^t\|]\leq \frac{G}{\sigma} a_M^t.$$
Let us now define
\begin{equation*}
    \bzl^t\triangleq\frac{1}{N}\sum_{i=1}^N z_{\ell,i}^t.
\end{equation*}
Since $\min_{j}z_{\ell,j}^t\leq\bzl^t\leq\max_{j}z_{\ell,j}^t$, for all $t$, we have that
\begin{equation*}
    |z_{\ell,i}^t-\bzl^t|\leq \max_{j}z_{\ell,j}^t-\min_{j}z_{\ell,j}^t
\end{equation*}
for all $i\in\until{N}$.
Notice now that, by definition $x_{i,\ell}^t=z_{\ell,i}^t$ and $\bx_\ell^t=\bzl^t$. Hence,
\begin{align*}
    \E&[|x_{i,\ell}^{t}-\bx_\ell^{t}|] \leq M \left((\mu_{\ell})^{t-1} \E[\|\zl^0\|]+ \frac{G}{\sigma}  \sum_{s=0}^{t-2}(\mu_{\ell})^{t-s-2} a_M^s + \frac{G}{\sigma}  a_M^{t-1}\right).
\end{align*}
Finally, since $\|x_i^t-\bx^t\|\leq\sum_{\ell=1}^B|x_{i,\ell}^t-\bx_\ell^t|$, one has
\begin{align*}
    \E[\|x_i^{t}-\bx^{t}\|]&\leq \sum_{\ell=1}^B M\Bigg( (\mu_{\ell})^{t-1} \E[\|\zl^0\|]+ \frac{G}{\sigma}  \sum_{s=0}^{t-2}(\mu_{\ell})^{t-s-2}a_M^s + \frac{G}{\sigma}  a_M^{t-1}\Bigg)\nonumber\\
    &\leq MB\Bigg((\mu_M)^{t-1}\E[\|\x^0\|] +\frac{G}{\sigma} \sum_{s=0}^{t-2}(\mu_{M})^{t-s-2} a_M^s + \frac{G}{\sigma}  a_M^{t-1}\Bigg),
\end{align*}
where $\mu_M=\max_{\ell}\mu_\ell$. The proof is concluded by using Assumption~\ref{assumption:random_variables}\ref{assumption:initial}.
\end{proof}
Moreover, the expected value of the distance between $y_i^t$ and $x_i^t$ can be bounded, by exploiting the convexity of the norm and using Lemma~\ref{lemma:xi-bx}, as stated in the next result.
\begin{lemma}\label{lemma:yx}
    Let Assumptions~\ref{assumption:problem_structure}\ref{assumption:subgradients},~\ref{assumption:communication},~\ref{assumption:random_variables} hold. Then, 
    \begin{align*}
        \E[\|y_i^{t}-x_i^t\|]&\leq 2MB\Bigg((\mu_M)^{t-1}C+\frac{G}{\sigma}\sum_{s=0}^{t-2}(\mu_{M})^{t-s-2} a_M^s+\frac{G}{\sigma} a_M^{t-1}\Bigg)
    \end{align*}
    for all $i\in\until{N}$ and all $t\geq 1$.
\end{lemma}
\begin{proof}
    Form the definition of $y_i^t$ and using the convexity of the norm, one has
    \begin{align*}
        \|y_i^{t}-x_i^t\|&=\|\sum_{j=1}^N w_{ij}x_j^{t}-x_i^t\|\\
        &\leq \sum_{j=1}^N w_{ij}\|x_j^{t}-x_i^t\|\nonumber\\
        &\leq \sum_{j=1}^N w_{ij} \left(\|x_j^{t}-\bx^t\|+\|x_i^{t}-\bx^t\|\right).
    \end{align*}
    By taking the expected value on both sides and using Lemma~\ref{lemma:xi-bx}, the proof follows by noting that $\sum_{j=1}^N w_{ij}=1$ from Assumption~\ref{assumption:communication}\ref{assumption:stochastic}.
\end{proof}

\subsubsection{Constant stepsize}
The following two results respectively provide an upper bound on the distance of $x_i^t$ from $\bx^t$ as $t\to\infty$ and characterize the quantity $\sum_{\tau=0}^t \E[\|x_i^\tau-\bx^\tau\|]$ for each $t$ in the case of constant stepsizes.
\begin{lemma}\label{lemma:xi-bx_constant}
    Let Assumptions~\ref{assumption:problem_structure}\ref{assumption:subgradients},~\ref{assumption:communication},~\ref{assumption:constant},~\ref{assumption:random_variables} hold. Then, there exist constants $M\in(0,\infty)$
    and $\mu_M\in(0,1)$ such that
  \begin{equation*}
      \lim_{t\to\infty}\E[\|x_i^{t}-\bx^{t}\|] \leq \bS
  \end{equation*}
  for all $i\in\until{N}$, with
  \begin{equation}
      \bS=a_M \frac{MBG}{\sigma}\frac{2-\mu_M}{1-\mu_M}.
  \end{equation}
\end{lemma}
\begin{proof}
    Equation~\eqref{eq:xi-bx} in Lemma~\ref{lemma:xi-bx} consists of three terms. For the first one,
  $\lim_{t\to\infty}(\mu_M)^{t-1}C = 0$, since
  $\mu_M<1$. For the second term $\sum_{s=0}^{t-2}(\mu_{M})^{t-s-2} a_M^s$, by
  Assumption~\ref{assumption:constant} and Lemma~\ref{lemma:series_nedic}, one
  has that $\lim_{t\to\infty}\sum_{s=0}^{t-2}(\mu_{M})^{t-s-2} a_M^s=\frac{a_M}{1-\mu_M}$. The proof is completed by noting that, under Assumption~\ref{assumption:constant} the last term is constant. 
\end{proof}
\begin{lemma}\label{lemma:x_constant}
    Let Assumptions~\ref{assumption:problem_structure}\ref{assumption:subgradients},~\ref{assumption:communication},~\ref{assumption:constant},~\ref{assumption:random_variables} hold.  Then, 
    \begin{align*}
        \sum_{\tau=0}^t&  \E[\|x_i^{\tau}-\bx^{\tau}\|]\leq  (\mu_M)^t \bR + t \bS + \bQ
    \end{align*}
    for all $i\in\until{N}$, with
    \begin{align}
        \bR&= MB\left(\frac{a_M G}{\sigma(1-\mu_M)^2}-\frac{C}{1-\mu_M} \right),\\
        \bQ&= C-a_M\frac{MBG}{\sigma}\frac{1}{(1-\mu_M)^2}.
    \end{align}

\end{lemma}
\begin{proof}
    By using Assumption~\ref{assumption:random_variables}\ref{assumption:initial}, for $\tau=0$, one has
    \begin{align}\label{eq:x0}
        \E[\|x_i^{0}-\bx^{0}\|]&\leq  \E[\|x_i^{0}\|] +\E[\|\bx^{0}\|]\nonumber\\
        &\leq C_i +\frac{1}{N}\sum_{j=1}^N C_j \leq C_i +\max_j C_j \leq C
    \end{align}
    Hence, $\sum_{\tau=0}^t  \E[\|x_i^{\tau}-\bx^{\tau}\|]\leq C +\sum_{\tau=1}^t  \E[\|x_i^{\tau}-\bx^{\tau}\|]$ and, from Lemma~\ref{lemma:xi-bx}, we have 
    \begin{align*}
        \sum_{\tau=0}^t  \E[\|x_i^{\tau}-\bx^{\tau}\|]
        &\leq C + \sum_{\tau=1}^t  MB\Bigg((\mu_M)^{\tau-1}C+\frac{G}{\sigma}\sum_{s=0}^{\tau-2}(\mu_{M})^{\tau-s-2} a_M+\frac{G}{\sigma} a_M\Bigg)\\
        &= C + MBC\sum_{\tau=0}^{t-1}(\mu_M)^{\tau}+\frac{MBGa_M}{\sigma}\sum_{\tau=2}^t\sum_{s=0}^{\tau-2}(\mu_{M})^{\tau-s-2} + \sum_{\tau=0}^{t-1}\frac{MBG a_M}{\sigma}
    \end{align*}
    where in the last line we have rearranged the summations. Now, by noting that 
    $$
    \sum_{\tau=2}^t\sum_{s=0}^{\tau-2}(\mu_{M})^{\tau-s-2}=\sum_{\tau=0}^{t-2}\sum_{s=0}^{t-2-s}(\mu_{M})^{s}
    $$
    and using Lemma~\ref{lemma:series} the result follows through straightforward manipulations.
\end{proof}
Notice that in virtue of Lemma~\ref{lemma:yx}, by using the same reasoning used in the previous two results it is possible to show that
\begin{equation*}
    \lim_{t\to\infty}\E[\|y_i^{t}-x_i^{t}\|] \leq 2\bS
\end{equation*}
and
\begin{align*}
    \sum_{\tau=0}^t &\E[\|y_i^{\tau}-x_i^{\tau}\|]\leq 2\bigg((\mu_M)^t \bR + t \bS + \bQ)\bigg).
\end{align*}
These bounds will be used in the following optimality analysis.

\subsubsection{Diminishing stepsize}
When adopting diminishing stepsizes, asymptotic (exact) consensus can be reached as stated in the following result. 
\begin{lemma}\label{lemma:asymptotic_consensus}
    Let Assumptions~\ref{assumption:problem_structure}\ref{assumption:subgradients},~\ref{assumption:communication},~\ref{assumption:stepsize},~\ref{assumption:random_variables} hold. Then, 
\begin{equation*}
   \lim_{t\to\infty} \E[\|x_i^t-\bx^t\|]=0
\end{equation*}
for all $i\in\until{N}$.
\end{lemma} 
\begin{proof}
  The proof is based on the same arguments as in Lemma~\ref{lemma:xi-bx_constant}. By noting that $ a_M^t\to 0$ as $t\to\infty$ and that, under Assumption~\ref{assumption:stepsize}, from Lemma~\ref{lemma:series_nedic}, $\lim_{t\to\infty}\sum_{s=0}^{t-2}(\mu_{M})^{t-s-2} a_M^s=0$, the result follows.
\end{proof}
The next result shows that
$\lim_{t\to\infty} \sum_{\tau=0}^t a_m^{\tau}\E[\|x_i^\tau-\bx^\tau\|]$ is a summable series for all $i\in\until{N}$. Notice that this result does not hold in the case of constant stepsizes.
\begin{lemma}\label{lemma:xm}
    Let Assumptions~\ref{assumption:problem_structure}\ref{assumption:subgradients},~\ref{assumption:communication},~\ref{assumption:stepsize},~\ref{assumption:random_variables} hold.  Then, 
    \begin{equation*}
        \lim_{t\to\infty}\sum_{\tau=0}^t a_m^{\tau}\E[\|x_i^{\tau}-\bx^{\tau}\|]< \infty
    \end{equation*}
    for all $i\in\until{N}$.
\end{lemma}
\begin{proof}
    As for $\tau=0$, from~\eqref{eq:x0} we have $a_m^{0}\E[\|x_i^{0}-\bx^0\|]\leq a_m^0 C$, while, for $\tau\geq 1$, from Lemma~\ref{lemma:xi-bx} it holds that
    \begin{align*}
         a_m^{\tau}\E[\|x_i^{\tau}-\bx^{\tau}\|] &\leq a_m^{\tau} MB\Bigg((\mu_M)^{\tau-1}C + \frac{G}{\sigma}\sum_{s=0}^{\tau-2}(\mu_{M})^{\tau-s-2} a_M^s+\frac{G}{\sigma} a_M^{\tau-1}\Bigg).
    \end{align*}
    Since, by Assumption~\ref{assumption:stepsize}, $ a_m^{\tau+1}\leq a_m^\tau\leq a_M^\tau$ for all $\tau$, one has
    \begin{align*}
         a_m^{\tau}\E[\|x_i^{\tau}-\bx^{\tau}\|]\leq  MB\Bigg( a_m^{\tau}(\mu_M)^{\tau-1}C+ \frac{G}{\sigma}\sum_{s=0}^{\tau-2}(\mu_{M})^{\tau-s-2}(a_M^s)^2+\frac{G}{\sigma}( a_M^{\tau-1})^2\Bigg)
    \end{align*}
    and then,
    \begin{align*}
        \sum_{\tau=0}^t a_m^{\tau}\E[\|x_i^\tau-\bx^\tau\|]\leq a_m^0 C + MB\Bigg(C\sum_{\tau=1}^t (\mu_M)^{\tau-1} a_m^{\tau}+\frac{G}{\sigma}\sum_{\tau=1}^t\sum_{s=0}^{\tau-2}\mu_{M}^{\tau-s-2}(a_M^s)^2+ \frac{G}{\sigma}\sum_{\tau=1}^t (a_M^{\tau-1})^2\Bigg).
    \end{align*}
    Now, from Assumption~\ref{assumption:random_variables}\ref{assumption:initial}, $C<\infty$. Moreover, by using Assumption~\ref{assumption:stepsize} and Lemma~\ref{lemma:series_nedic}, we have $\lim_{t\to\infty}\sum_{\tau=1}^t (\mu_M)^{\tau-1} a_m^\tau<\infty$. 
    Finally, by Assumption~\ref{assumption:stepsize}, $\sum_{\tau=0}^\infty( a_M^\tau)^2<\infty$, so that, from Lemma~\ref{lemma:series_nedic}, $\lim_{t\to\infty}\sum_{\tau=1}^t\sum_{s=0}^{\tau-2}(\mu_{M})^{\tau-s-2}( a_M^s)^2{<}\infty$, thus concluding the proof.
\end{proof}
As in the case of constant stepsizes, thanks to Lemma~\ref{lemma:yx}, it can be shown that $\lim_{t\to\infty} \E[\|y_i^t-x_i^t\|]=0$ and $\lim_{t\to\infty}\sum_{\tau=0}^t a_m^{\tau}\E[\|y_i^{\tau}-x_i^\tau\|]< \infty$.

\subsection{Optimality}
In this section, we show the convergence of the \DBPnospace.
First, a bound on the expected distance from the optimal cost at iteration $t$ is given without any assumption on the stepsize sequence. Then, it is shown that such a distance goes to $0$ as $t\to\infty$ for diminishing stepsizes, while it is upper bounded by a finite quantity for constant stepsizes and an explicit convergence rate is provided.

We start by defining the Ljapunov function
\begin{equation}\label{eq:V}
    V_i^\tau \triangleq \sum_{\ell=1}^B \pi_{i,\ell}^{-1}\nu_{\ell}(x_{i,\ell}^{\tau},x_{\ell}^\star)
\end{equation}
and $V^t \triangleq \sum_{i=1}^N V_i^t$.
Moreover, given a sequence of points $\{z^\tau\}_{\tau=0}^t$, we define
\begin{equation}
    \ff(z^t)\triangleq\min_{\tau\leq t} \E[f(z^\tau)]
\end{equation}
Then, the following result holds true.
\begin{theorem}\label{theorem:bound}
    Let Assumptions~\ref{assumption:problem_structure},~\ref{assumption:communication},~\ref{assumption:separate},~\ref{assumption:random_variables} hold. 
    Then,
    \begin{align}
        \ff(\bx^t)-f(x^\star)&\leq\left(\sum_{\tau=0}^t a_m^{\tau}\right)^{-1}\Bigg(\E[V^0]+\sum_{\tau=0}^t\frac{( a_M^{\tau})^2 \bG}{2\sigma}\nonumber\\
        & \hspace{3ex} + \sum_{\tau=0}^t a_m^{\tau} \sum_{i=1}^N G_i\E[\|y_i^{\tau}-x_i^{\tau}\|]\nonumber\\
        & \hspace{3ex} + \sum_{\tau=0}^t a_m^{\tau} \sum_{i=1}^N G_i\E[\|x_i^{\tau}-\bar{x}^{\tau}\|]\Bigg).\label{eq:f_bound}
    \end{align}
\end{theorem}
\begin{proof}
    In order to simplify the notation, let us denote $\gs_i^\tau=\gs_i(y_i^{\tau})$ and $g_{i,\ell_i^\tau}=g_{i,\ell_i^\tau}(y_i^{\tau},\xi_i^\tau)$.
    From the convexity of $f$ we have that, at a given iteration $t$,
    \begin{align}
        \left(\sum_{\tau=0}^t a_m^{\tau}\right)(\ff(\bx^t)-f(x^\star)) &=\left(\sum_{\tau=0}^t a_m^{\tau}\right)(\min_{\tau\leq t} (\E[f(\bar{x}^\tau)]-f(x^\star))\nonumber\\
        &\leq \sum_{\tau=0}^{t}a_m^{\tau}\left( \E[f(\bar{x}^{\tau})]-f(x^\star) \right).\label{eq:xstar}
    \end{align}
Now, we make some manipulation on the term $\E[f(\bar{x}^{\tau})]-f(x^\star)=\E[f(\bx^{\tau})-f(x^\star)]$:
\begin{align}
    \E[f(\bar{x}^{\tau})-f(x^\star)]&=\sum_{i=1}^N \E[\left(f_i(\bar{x}^{\tau})-f_i(x^\star)\right)]\nonumber\\
    &=\sum_{i=1}^N\E[\left(f_i(x_i^{\tau})-f_i(x^\star)+f_i(\bar{x}^{\tau})-f_i(x_i^{\tau})\right)] \nonumber\\
    &\leq\sum_{i=1}^N \E[\left(f_i(x_i^{\tau})-f_i(x^\star)\right)] +\sum_{i=1}^N G_i\E[\|x_i^{\tau}-\bar{x}^{\tau}\|] \nonumber\\
    &=\sum_{i=1}^N \E[\left(f_i(y_i^{\tau})-f_i(x^\star)+f_i(x_i^{\tau})-f_i(y_i^{\tau})\right)] +\sum_{i=1}^N G_i\E[\|x_i^{\tau}-\bar{x}^{\tau}\|] \nonumber\\
    &\leq\sum_{i=1}^N \E[\left(f_i(y_i^{\tau})-f_i(x^\star)\right)] {+\sum_{i=1}^N G_i\left(\E[\|y_i^{\tau}-x_i^{\tau}\|]+\E[\|x_i^{\tau}-\bar{x}^{\tau}\|]\right) },\label{eq:x_i-xstar}
\end{align}
where we have used multiple times the convexity of $f$ (and of each $f_j$) and the subgradient boundedness (Assumption~\ref{assumption:problem_structure}\ref{assumption:subgradients}).
Let us now study the term $\sum_{i=1}^N \E[f_i(y_i^{\tau})-f_i(x^\star)]$ in~\eqref{eq:x_i-xstar}.
By writing the optimality condition for the proximal mapping in the update~\eqref{eq:x_update}, if $s_i^\tau=1$, one has
\begin{align}
    \nu_{\ell_i^\tau}(x_{i,\ell_i^\tau}^{\tau+1},x_{\ell_i^\tau}^\star)&\leq \nu_{\ell_i^\tau}(y_{i,\ell_i^\tau}^{\tau},x_{\ell_i^\tau}^\star)-\alpha_i^\tau\langle U_{\ell_i^{\tau}}g_i^{\tau},y_i^{\tau}-x^\star\rangle+\frac{(\alpha_i^\tau)^2}{2\sigma} \|g_{i,\ell_i^\tau}\|^2.\label{eq:optimality}
\end{align}
Hence, 
\begin{equation}
    \nu_{\ell}(x_{i,\ell}^{\tau+1},x_{\ell}^\star)\leq
    \begin{cases} 
        \eqref{eq:optimality},&\text{if }\ell=\ell_i^\tau,\text{ and } s_i^\tau =1\\
        \nu_{\ell}(x_{i,\ell}^{\tau},x_{\ell}^\star),&\text{otherwise}.
    \end{cases}\label{eq:tx_update}
\end{equation}
Thus, from~\eqref{eq:tx_update} and by using~\eqref{eq:V}, one has that, if $s_i^\tau=1$, it holds
\begin{align}
    V_i^{\tau+1} &\leq \sum_{m\neq\ell_i^\tau} \pi_{i,m}^{-1}\nu_{m}(x_{i,m}^{\tau},x_{m}^\star) + \pi_{i,\ell_i^\tau}^{-1}\Bigg(\nu_{\ell_i^\tau}(y_{i,\ell_i^\tau}^{\tau},x_{\ell_i^\tau}^\star) -\alpha_i^\tau\langle U_{\ell_i^{\tau}}g_i^{\tau},y_i^{\tau}-x^\star\rangle+\frac{(\alpha_i^\tau)^2}{2\sigma} \|g_{i,\ell_i^\tau}\|^2\Bigg),\label{eq:von}
\end{align}
while, if $s_i^\tau=0$,
\begin{equation}
    V_i^{\tau+1} = V_i^{\tau} = \sum_{\ell=1}^B \pi_{i,\ell}^{-1}\nu_{\ell}(x_{i,\ell}^{\tau},x_{\ell}^\star).\label{eq:voff}
\end{equation}
Now, by taking the expected value of $V_i^{\tau+1}$ conditioned to $\mS^{\tau}$, one obtains
\begin{align}
    \E[V_i^{\tau+1}\mid\mS^{\tau}]&=(1-p_{i,on})\E[V_i^{\tau+1}\mid\mS^{\tau}, s_i^\tau=0]+p_{i,on}\E[V_i^{\tau+1}\mid\mS^{\tau}, s_i^\tau=1],\label{eq:EV}
\end{align}
and hence, by substituting~\eqref{eq:von} and~\eqref{eq:voff} in~\eqref{eq:EV},
\begin{align}
    \E[V_i^{\tau+1}\mid\mS^{\tau}]
    &\leq(1-p_{i,on})\sum_{\ell=1}^B \pi_{i,\ell}^{-1}\nu_{\ell}(x_{i,\ell}^{\tau},x_{\ell}^\star) \nonumber\\
    &\hspace{3ex} +p_{i,on}\sum_{\ell=1}^B p_{i,\ell} \Bigg(\sum_{m\neq\ell}\pi_{i,m}^{-1}\nu_{m}(x_{i,m}^{\tau},x_{m}^\star) \nonumber\\
    &\hspace{3ex} +  \pi_{i,\ell}^{-1}\nu_{\ell}(y_{i,\ell}^{\tau},x_{\ell}^\star)-\alpha_i^\tau \pi_{i,\ell}^{-1} \E[\langle U_{\ell}g_i^\tau,y_i^{\tau}-x^\star\rangle]\nonumber\\
    &\hspace{3ex} +\frac{(\alpha_i^\tau)^2}{2\sigma} \pi_{i,\ell}^{-1} \E[\|g_{i,\ell}\|^2] \Bigg).\label{eq:vi1}
\end{align}
Notice now that
\begin{align}
    \sum_{\ell=1}^B p_{i,\ell} \sum_{m\neq\ell}\pi_{i,m}^{-1}\nu_{m}(x_{i,m}^{\tau},x_{m}^\star)&=\sum_{\ell=1}^B \left(p_{i,\ell}\sum_{m=1}^B\pi_{i,m}^{-1}\nu_{m}(x_{i,m}^{\tau},x_{m}^\star) - p_{i,\ell} \pi_{i,\ell}^{-1}\nu_{\ell}(x_{i,\ell}^{\tau},x_{\ell}^\star)\right)\nonumber\\
    &=\sum_{m=1}^B \sum_{\ell=1}^B p_{i,\ell}\pi_{i,m}^{-1}\nu_{m}(x_{i,m}^{\tau},x_{m}^\star)- \sum_{\ell=1}^Bp_{i,\ell} \pi_{i,\ell}^{-1}\nu_{\ell}(x_{i,\ell}^{\tau},x_{\ell}^\star)\nonumber\\
    &= \sum_{m=1}^B \pi_{i,m}^{-1}\nu_{m}(x_{i,m}^{\tau},x_{m}^\star)- \sum_{\ell=1}^Bp_{i,\ell} \pi_{i,\ell}^{-1}\nu_{\ell}(x_{i,\ell}^{\tau},x_{\ell}^\star)\nonumber\\
    &=\sum_{\ell=1}^B (1-p_{i,\ell}) \pi_{i,\ell}^{-1}\nu_{\ell}(x_{i,\ell}^{\tau},x_{\ell}^\star).\label{eq:pil}
\end{align}
Moreover, by noting that it holds that $\sum_{\ell=1}^B \E[\|g_{i,\ell}^\tau\|^2] =  \E[\sum_{\ell=1}^B\|g_{i,\ell}^\tau\|^2]=\E[\|g_i^\tau\|^2]$ and $\sum_{\ell=1}^B U_\ell g_i^\tau = g_i^\tau$ and by substituting~\eqref{eq:pil} in~\eqref{eq:vi1}, we obtain
\begin{align}
    \E[V_i^{\tau+1}\mid\mS^{\tau}]
    &= \sum_{\ell=1}^B\left((1-p_{i,on})+p_{i,on}-p_{i,on}p_{i,\ell}\right)\pi_{i,\ell}^{-1}\nu_{\ell}(x_{i,\ell}^{\tau},x_{\ell}^\star)\nonumber\\
    &\hspace{3ex} + \sum_{\ell=1}^B\nu_{\ell}(y_{i,\ell}^{\tau},x_{\ell}^\star)-\alpha_i^\tau  \E[\langle g_i^\tau,y_i^{\tau}-x^\star\rangle]+\frac{(\alpha_i^\tau)^2}{2\sigma} \E[\|g_i^\tau\|^2] \nonumber\\
    &= \sum_{\ell=1}^B(1-\pi_{i,\ell})\pi_{i,\ell}^{-1}\nu_{\ell}(x_{i,\ell}^{\tau},x_{\ell}^\star)+ \sum_{\ell=1}^B\nu_{\ell}(y_{i,\ell}^{\tau},x_{\ell}^\star)\nonumber\\
    &\hspace{3ex} -\alpha_i^\tau  \E[\langle g_i^\tau,y_i^{\tau}-x^\star\rangle]+\frac{(\alpha_i^\tau)^2}{2\sigma} \E[\|g_i^\tau\|^2] \nonumber\\
    &= V_i^\tau - \sum_{\ell=1}^B\nu_{\ell}(x_{i,\ell}^{\tau},x_{\ell}^\star) + \sum_{\ell=1}^B\nu_{\ell}(y_{i,\ell}^{\tau},x_{\ell}^\star)\nonumber\\
    &\hspace{3ex}  -\alpha_i^\tau  \langle \E[g_i^\tau],y_i^{\tau}-x^\star\rangle+\frac{(\alpha_i^\tau)^2}{2\sigma} \E[\|g_i^\tau\|^2] \nonumber\\
    &\leq V_i^{\tau} - \sum_{\ell=1}^B\nu_{\ell}(x_{i,\ell}^{\tau},x_{\ell}^\star) +
    \sum_{j=1}^N w_{ij}\sum_{\ell=1}^B\nu_{\ell}(x_{j,\ell}^{\tau},x_{\ell}^\star)\nonumber\\
    &\hspace{3ex}  -\alpha_i^\tau  \langle \gs_i^\tau,y_i^{\tau}-x^\star\rangle+\frac{(\alpha_i^\tau)^2 \bG_i}{2\sigma}
    \label{eq:start2}
\end{align} 
where in the last inequality we used the separate convexity of $\nu_{\ell}$ from Assumptions~\ref{assumption:separate} and the fact that $\E[\|g_i^\tau\|^2]\leq\bG_i$ (Assumption~\ref{assumption:problem_structure}\ref{assumption:subgradients}), and $\E[g_i^\tau]=\gs_i^\tau$ (Assumption~\ref{assumption:problem_structure}\ref{assumption:unbiased}).
Now, by summing over $i$, 
\vspace{-1ex}
\begin{align}
    \sum_{i=1}^N \E[V_i^{\tau+1}\mid\mS^{\tau}]
    &\leq \sum_{i=1}^N V_i^{\tau} - \sum_{i=1}^N \sum_{\ell=1}^B\nu_{\ell}(x_{i,\ell}^{\tau},x_{\ell}^\star) + \sum_{i=1}^N\sum_{j=1}^N w_{ij}\sum_{\ell=1}^B\nu_{\ell}(x_{j,\ell}^{\tau},x_{\ell}^\star)\nonumber\\
    &\hspace{3ex}-\sum_{i=1}^N\alpha_i^\tau \langle \gs_i^\tau,y_i^{\tau}-x^\star\rangle+\sum_{i=1}^N\frac{(\alpha_i^\tau)^2 \bG_i}{2\sigma}\nonumber\\
    &= \sum_{i=1}^N V_i^{\tau} - \sum_{i=1}^N \sum_{\ell=1}^B\nu_{\ell}(x_{i,\ell}^{\tau},x_{\ell}^\star) + \sum_{j=1}^N \sum_{\ell=1}^B\nu_{\ell}(x_{j,\ell}^{\tau},x_{\ell}^\star)\nonumber\\
    &\hspace{3ex}-\sum_{i=1}^N\alpha_i^\tau \langle \gs_i^\tau,y_i^{\tau}-x^\star\rangle+\sum_{i=1}^N\frac{(\alpha_i^\tau)^2 \bG_i}{2\sigma}\nonumber\\
    &= \sum_{i=1}^N V_i^{\tau} -\sum_{i=1}^N\alpha_i^\tau \langle \gs_i^\tau,y_i^{\tau}-x^\star\rangle+\sum_{i=1}^N\frac{(\alpha_i^\tau)^2 \bG_i}{2\sigma}\nonumber\\
    &\leq \sum_{i=1}^N V_i^{\tau} -\sum_{i=1}^N\alpha_i^\tau \left(f_i(y_i^{\tau})-f_i(x^\star)\right)+\frac{(a_M^\tau)^2 \bG}{2\sigma},\label{eq:sumV}
\end{align}
where in the last inequality we used the convexity of $f_i$. 
Taking the expected value conditioned to $\mS^{0}$ on both sides of~\eqref{eq:sumV}, using $\E[V_i^{\tau+1}{\mid}\mS^{\tau}]{=}\E[V_i^{\tau+1}{\mid}\mS^{\tau},\mS^0]$ and  Lemma~\ref{lemma:expectation}, gives
\begin{align*}
    \sum_{i=1}^N \E\left[\E[V_i^{\tau+1}\mid\mS^{\tau}]\mid\mS^{0}\right]&=\sum_{i=1}^N \E[V_i^{\tau+1}\mid\mS^{0}]\nonumber\\
    &\leq \sum_{i=1}^N\E[V_i^{\tau}\mid\mS^{0}]-\sum_{i=1}^N\alpha_i^\tau \left(\E[f_i(y_i^{\tau})\mid\mS^{0}]-f_i(x^\star)\right) +\frac{(a_M^\tau)^2 \bG}{2\sigma}.
\end{align*}
and, rearranging the terms,
\begin{align*}
    \sum_{i=1}^N&\alpha_i^\tau\left(\E[f_i(y_i^{\tau})\mid\mS^0]-f_i(x^\star)\right)\leq\sum_{i=1}^N\E[V_i^{\tau}\mid\mS^0]-\sum_{i=1}^N \E[V_i^{\tau+1}\mid\mS^0]+\frac{(a_M^\tau)^2 \bG}{2\sigma}.
\end{align*}
Now, by summing over $\tau$, and noting that $\E[V_i^0\mid\mS^0]=V_i^0$,
\begin{align*}
    \sum_{\tau=0}^{t} \sum_{i=1}^N \alpha_i^{\tau} \left(\E[f_i(y_i^{\tau})\mid\mS^0]-f_i(x^\star)\right)&\leq\sum_{i=1}^NV_i^0-\sum_{i=1}^N \E[V_i^{t+1}\mid\mS^0]+\sum_{\tau=0}^t\frac{(a_M^\tau)^2 \bG}{2\sigma}\nonumber\\
    &\leq V^{0}+\sum_{\tau=0}^t\frac{(a_M^\tau)^2 \bG}{2\sigma}.
\end{align*}
Moreover, by taking the expected value over $\mS^0=\{\x^0\}$,
\begin{align}
    \sum_{\tau=0}^{t} \sum_{i=1}^N \alpha_i^{\tau} \left(\E[f_i(y_i^{\tau})]-f_i(x^\star)\right)
    &\leq \E[V^{0}]+\sum_{\tau=0}^t\frac{(a_M^\tau)^2 \bG}{2\sigma}.\label{eq:ys}
\end{align}
Notice now that, since by definition $a_m^\tau\leq \alpha_i^\tau$ for all $i$ and $\tau$, we have
\begin{align}
    \sum_{\tau=0}^{t} a_m^{t}& \sum_{i=1}^N \left(\E[f_i(y_i^{\tau})]-f_i(x^\star)\right)\leq \sum_{\tau=0}^{t} \sum_{i=1}^N \alpha_i^{\tau} \left(\E[f_i(y_i^{\tau})]-f_i(x^\star)\right).\label{eq:aa}
\end{align}
Finally, by combining~\eqref{eq:xstar},~\eqref{eq:x_i-xstar},~\eqref{eq:ys} and~\eqref{eq:aa} one obtains~\eqref{eq:f_bound}.
\end{proof}
A similar result can be given also in terms of the sequences of local solution estimates $\{x_i^t\}$ as formalized in the next result.
 \begin{corollary}\label{corollary:bound}
     Let Assumptions~\ref{assumption:problem_structure},~\ref{assumption:communication},~\ref{assumption:separate},~\ref{assumption:random_variables} hold. 
     Then,
     \begin{align}
       \ff(x_i^t)-f(x^\star)&\leq\left(\sum_{\tau=0}^t a_m^{\tau}\right)^{-1}\Bigg(\E[V^0]+\sum_{\tau=0}^t\frac{( a_M^{\tau})^2 \bG}{2\sigma}\nonumber\\
        & \hspace{3ex} + \sum_{\tau=0}^t a_m^{\tau} \sum_{j=1}^N G_j\E[\|y_j^{\tau}-x_j^{\tau}\|]\nonumber\\
        & \hspace{3ex} + \sum_{\tau=0}^t a_m^{\tau} \sum_{j=1}^N G_j\E[\|x_j^{\tau}-\bar{x}^{\tau}\|]\nonumber\\
        & \hspace{3ex} + \sum_{\tau=0}^t a_m^{\tau} G \E[\|x_i^{\tau}-\bar{x}^{\tau}\|]\Bigg).\label{eq:f_i_bound}
     \end{align}
     for all $i\in\until{N}$.
 \end{corollary}
 \begin{proof}
     The proof follows the same line of the one of Theorem~\ref{theorem:bound}, by noting that
     \begin{align*}
         f(x_i^t)-f(x^\star)&=f(x_i^t)-f(\bx^t)+f(\bx^t)-f(x^\star)\\
         &\leq f(\bx^t)-f(x^\star) + G\|x_i^t-\bx^t\|,
     \end{align*}
    so that in place of~\eqref{eq:xstar}, one has
     \begin{align*}
        \left(\sum_{\tau=0}^t a_m^{\tau}\right)(\ff(x_i^t)-f(x^\star)) &\leq \sum_{\tau=0}^{t}a_m^{\tau}\left( \E[f(x_i^{\tau})]-f(x^\star) \right)\nonumber\\
        &\leq \sum_{\tau=0}^{t}a_m^{\tau}\left( \E[f(\bar{x}^{\tau})]-f(x^\star) + G\E[\|x_i^\tau-\bx^\tau\|]\right)    \end{align*}
    for all $i\in\until{N}$.
 \end{proof}
The previous two results hold true without making any assumption on the local stepsize sequences.
In the next two subsections the general result of Theorem~\ref{theorem:bound} is specialized to the case of constant and diminishing stepsizes respectively. 
 
\subsubsection{Constant stepsizes}\label{subsec:rate}
In the case of constant (local) stepsizes, convergence with a constant error is attained with an explicit sublinear convergence rate.
\begin{theorem}\label{theorem:bound_constant}
    Let Assumptions~\ref{assumption:problem_structure},~\ref{assumption:communication},~\ref{assumption:separate},~\ref{assumption:constant},~\ref{assumption:random_variables} hold. 
    Then, there exist constants $M\in(0,\infty)$ and $\mu_M\in(0,1)$ such that
    \begin{align}
        \ff(x_i^t)-f(x^\star)\leq \frac{Q+(\mu_M)^t R}{t+1} +  S\label{eq:bound_constant}
    \end{align}
    for all $i\in\until{N}$, with 
    \begin{align*}
        Q&=\frac{\E[V^0]}{a_m}+ 4G\left(\frac{MBC}{1-\mu_M}+\bQ\right),\\
        R&= 4G\bR,\\
        S&= 4G\bS +\frac{a_M^2}{a_m}\frac{\bG}{2\sigma}.
    \end{align*}
\end{theorem}
\begin{proof}
    By exploiting Assumption~\ref{assumption:constant}, Lemma~\ref{lemma:x_constant} and Lemma~\ref{lemma:yx}, from~\eqref{eq:f_i_bound} one obtains
    \begin{align*}
        \ff(x_i^t)-f(x^\star) &\leq\left(\sum_{\tau=0}^t a_m\right)^{-1}\Bigg(\E[V^0]+\sum_{\tau=0}^t\frac{( a_M)^2 \bG}{2\sigma}+ 4a_m G\bigg( (\mu_M)^t \bR + t \bS + \bQ\bigg)\Bigg)
    \end{align*}
    which, by rearranging the terms, leads to
    \begin{align*}
        \ff(x_i^t)-f(x^\star)&\leq \frac{Q+(\mu_M)^t R}{t+1}+  \frac{t}{t+1}4G\bS +\frac{a_M^2}{a_m}\frac{\bG}{2\sigma}\\
        &\leq \frac{Q+(\mu_M)^t R}{t+1} +  S
    \end{align*}
    thus concluding the proof.
\end{proof}

The previous result shows that, when constant stepsizes are employed, the value of $\ff(x_i^t)$ converges to $f(x^\star)$ plus a constant error, which can be retrieved from~\eqref{eq:bound_constant} by taking the limit for $t\to\infty$, i.e., $\lim_{t\to\infty} \ff(x_i^t)-f(x^\star) \leq S$ with the explicit expression for $S$ being
\begin{equation*}
	S = a_M\left(\frac{4MBG^2}{\sigma}\frac{2-\mu_M}{1-\mu_M} + \frac{a_M}{a_m}\frac{\bar{G}}{\sigma}\right).
\end{equation*}
It is worth noting that the bound decreases with the maximum stepsize $a_M$.
Regarding the convergence rate, it is sublinear $O\left(Q/t\right)$. 
However, in the first iterations, the term $(\mu_M)^t R$ can be dominant (if $|R|\gg Q$), thus leading to a linear rate at the beginning of the algorithm (as it will be shown in the numerical example). 
Notice that $Q$ (and hence the convergence rate) depends both on the number of blocks and on the local probabilities of being awake and drawing blocks. The local probabilities appear in $V^0$ and (implicitly) in the constant $\mu_M$. In fact, $\mu_M$ is related to the randomness of the matrices $W_\ell^t$, which depend on such probabilities. 
In particular notice that, if $B=1$ the rate is similar to those obtained in~\cite{nedic2009distributed}, in which the proximal mapping $\nu(a,b)=\frac{1}{2}\|a-b\|^2$ is used.
Clearly, one may argue that the best rate is achieved by using a single block and hence, communicating in terms of blocks is useless. This is true only if we assume an infinite bandwidth to be available in the communication channels (i.e, transmitting the entire optimization variable or a single block of it requires the same amount of time).
However, in typical real world scenarios this is not true and data that exceed the communication bandwidth are transmitted sequentially. 
If only one block fits the communication channel, our algorithm allows to perform an update at each communication round, while classical ones would need $B$ communication rounds per update.
Moreover, in the proposed algorithm, typically, the local computation time at each iteration is \emph{not} negligible, since a minimization problem is to be solved at every step (see~\eqref{eq:x_update_l}). Solving such an optimization problem on the entire optimization variable or on a single block of it clearly results in completely different computational times, which are clearly lower in the case of block-wise updates. 
Thus, the benefits of using block-wise updates and communications make the \DBP well suited for big-data optimization problems.

\subsubsection{Diminishing stepsizes}
In the case of diminishing (local) stepsizes, asymptotic convergence to the optimal cost can be reached.
\begin{theorem}\label{theorem:bound_diminishing}
    Let Assumptions~\ref{assumption:problem_structure},~\ref{assumption:communication},~\ref{assumption:separate},~\ref{assumption:stepsize},~\ref{assumption:random_variables} hold. 
    Then,
    \begin{equation*}
        \lim_{t\to\infty}\ff(x_i^t)-f(x^\star)=0,
    \end{equation*}
    for all $i\in\until{N}.$
\end{theorem}
\begin{proof}
    The proof follows by taking the limit for $t\to\infty$, and using Assumption~\ref{assumption:stepsize} and Lemma~\ref{lemma:xm} in~\eqref{eq:f_i_bound}.
\end{proof}
\begin{remark}
	We point out that, similarly to, e.g.,~\cite{duchi2012dual,dang2015stochastic}, one can introduce a running averaging mechanism by defining $\hat{x}_i^t=\tfrac{1}{t}\sum_{\tau=1}^t x_i^t$ and provide the convergence results in terms of $f(\hat{x}_i^t)-f(x^\star)$, in place of $\ff(x_i^t)-f(x^\star)$. The convergence proof would follow almost the same line with some adjustments in~\eqref{eq:xstar}-\eqref{eq:x_i-xstar} (see~\cite{dang2015stochastic}). However, running averaging mechanisms typically lead to a significantly slower convergence rate. In light of this, we provided our results in terms of $\ff(x_i^t)-f(x^\star)$.
\end{remark}
\section{Numerical example}\label{sec:experiment}
Consider a soft margin classification problem in which each agent $i\in\until{N}$ has $m_i$ training samples $q_i^1,\dots,q_{i}^{m_i}\in\R^d$ each of which has an associated binary label $b_{i}^r\in\{-1,1\}$ for all $r\in\until{m_i}$. The goal of the network is to build a linear classifier from the training samples, i.e., to find a hyperplane of the form $\{z\in\R^{d}\mid \langle \theta, z\rangle + \theta_0=0\}$, with $\theta\in\R^d$ and $\theta_0\in\R$, which better separates the training data. Let us define $x=[\theta^\top, \theta_0]^\top\in\R^{d+1}$ and $\hat{q}_{i}^r=[(q_{i}^r)^\top, 1]^\top$. Then, the solution to this problem can be determined by solving the following SVM problem
\begin{equation}\label{pb:regression}
    \begin{aligned}
        &\m_{x\in\R^{d+1}} 
        & & \sum_{i=1}^N\frac{1}{m_i}\sum_{r=1}^{m_i}\log\left(1+e^{-b_{i}^r\langle x,\hat{q}_{i}^r\rangle}\right)+\lambda\|x\|_1,
    \end{aligned}
\end{equation}
where $\lambda>0$ is the regularization weight. Problem~\eqref{pb:regression} can be written in the form of problem~\eqref{pb:problem} by defining $\xi_i^r=(\hat{q}_{i}^r,b_i^r)$ and 
\begin{align*}
    \E[h_i(x;\xi_i)]&=\frac{1}{m_i}\sum_{r=1}^{m_i} h_i(x;\xi_i^r)\\
    &=\frac{1}{m_i}\sum_{r=1}^{m_i}\bigg(\log\left(1+e^{-b_{i}^r\langle x,\hat{q}_{i}^r\rangle}\right)+\frac{\lambda}{N}\|x\|_1\bigg)
\end{align*}
for all $i\in\until{N}$. 
Notice that, as long as each data $\xi_i^r$ is uniformly drawn from the dataset, Assumption~\ref{assumption:problem_structure}\ref{assumption:unbiased} is satisfied.

In the next two sections we will test the \DBS in the presented scenario, first on a synthetic dataset and then on real-world dataset composed of text documents.
In both cases we consider a system with $N=48$ processors. The proposed distributed algorithm has been implemented by using the Python package DISROPT~\cite{farina2019disropt}, and each processor has been assigned an agent.

\subsection{Synthetic dataset}
In order to show how the algorithm performs for different number of blocks, let us consider a relatively low-dimensional problem with $x\in\R^{50}$ and evaluate the algorithm performance for different number of blocks, namely $B\in\{1,2,5,10,50\}$.
We generate a synthetic dataset (with $x\in\R^{50}$) composed of 240 points (taken from two clusters corresponding to labels $-1$ and $1$ respectively) and assign $5$ of them to each agent, i.e., $m_1=\dots=m_N=5$. 
Finally, we set $\lambda=0.1$, a common (constant) stepsize $\alpha=0.2$, $p_{i,\ell}=1/B$ for all $i$ and all $\ell$ and $s_i^t=1$ for all $i$ and all $t$.
Regarding the communication graph, it has is generated according to an Erd\H{o}s-R\`{e}nyi random model with connectivity parameter $p=0.3$. The corresponding weight matrix is built by using the Metropolis-Hastings rule.
The evolution of the cost error adjusted with respect to the number of blocks is reported in Figure~\ref{fig:cost} for the considered block numbers. The results confirm the discussion carried out in Section~\ref{subsec:rate} about the role of block communications. In fact, when normalizing the number of iterations with respect to the number of blocks, the convergence rates for the considered number of blocks are comparable.
Moreover, the convergence rate exhibits the properties shown in Section~\ref{subsec:rate}. In fact, it is linear at the beginning and becomes sublinear after some iterations. Moreover, as expected when using constant stepsizes, convergence is reached with a constant error.

\begin{figure}[t]
        \centering 
        \includegraphics[width=0.6\columnwidth]{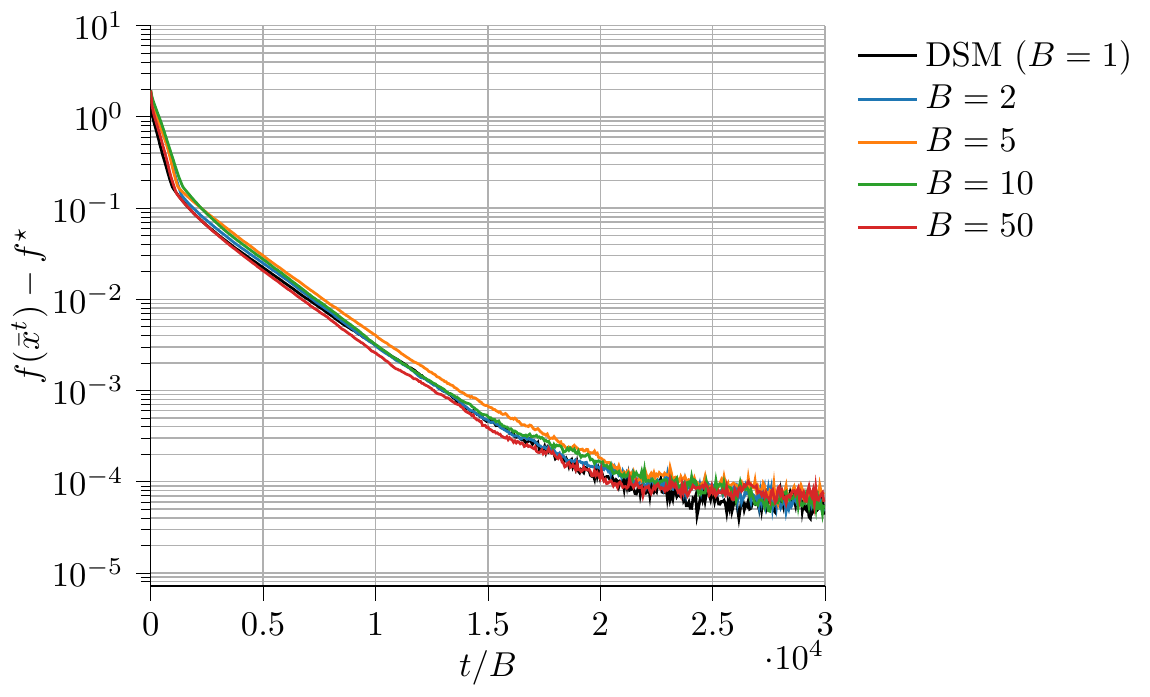}
        \caption{Numerical example: synthetic dataset. Evolution of the cost error for $B\in\{1,2,5,10,50\}$. The time scale is normalized on the number of blocks. The case $B=1$ coincides with the distributed subgradient method~\cite{nedic2009distributed}.}
        \label{fig:cost}
\end{figure} 

\subsection{Text classification}
Let us now consider a real-world scenario in which the local training samples are drawn from a dataset of texts. In particular, we pick the \emph{20 newsgroups dataset}, a dataset consisting of 18,846 newsgroup posts belonging to $20$ topics. Texts are represented by \textsc{tf-idf} on a dictionary of 130,107 words, so that each sample is a vector in $\R^{130,107}$. Agents have to learn to classify posts belonging to the class \emph{sci.med} from the others. Thus, in order to perform a binary classification, we assign the label $1$ to samples belonging to the class \emph{sci.med}, and $-1$ to all the other samples. In this scenario, the considered $48$ agents, are connected over a balanced directed graph, generated according to a binomial random model with connectivity parameter $p=0.5$ and each agent is awake with probability $p_{i,on}=0.95$. The entire dataset is split to assign almost the same number of samples to each agent. 
We run the algorithm for $4000$ iterations and for different number of blocks, namely $B\in\{10, 10^2,10^3, 10^4\}$. 
Moreover, we set $\lambda=0.001$ in problem~\eqref{pb:regression}, and we select a common (constant) stepsize $\alpha=0.5$ and $p_{i,\ell}=1/B$ for all $i$ and all $\ell$.
Differently from the previous example over synthetic data, in this case computing the exact (centralized) solution of the considered problem is computationally intractable, due to the high dimension of the decision variable ($n=$130,107) and the large number of samples ($\sum_{i=1}^N m_i=$18,146).
Thus, the performance of the algorithm are evaluated in terms of the accuracy $\Psi$ of the average of the produced solution estimates $\bar{x}^t$, i.e., the number of samples of the dataset that are correctly classified through the hyperplane defined by $\bar{x}^t$. The results are reported in Figure~\ref{fig:text_classification}.

\begin{figure}[!htb]
	\centering
	\includegraphics[width=0.6\columnwidth]{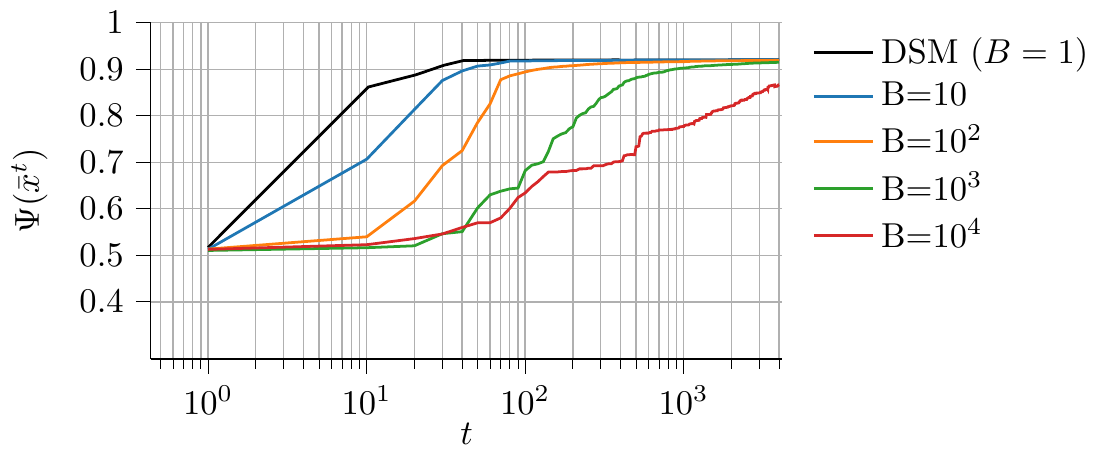}
	\caption{Numerical example: text classification. Evolution of the accuracy over the entire \emph{20 newsgroups} dataset.}
	\label{fig:text_classification}
\end{figure}

\section{Conclusions}\label{sec:conclusion}
In this paper, we introduced a class of distributed block proximal algorithms for solving stochastic big-data convex optimization problems over networks. In the addressed optimization set-up the dimension of the decision variable is very high and the (stochastic) cost function may be nonsmooth. The main strength of the proposed algorithms is that agents in the network can communicate a single block of the optimization variable per iteration.
Under the assumption of diminishing stepsizes, we showed that the agents in the network asymptotically agree on a common solution which is cost-optimal in expected value. When employing constant stepsizes approximate convergence is attained with a constant error on the optimal cost and an explicit convergence rate is provided. Special instances of the algorithm are presented for particular classes of problems.
Finally, the proposed algorithm, has been numerically evaluated on a distributed classification problem over both a synthetic dataset and a real, high-dimensional, text document dataset.

\bibliographystyle{IEEEtran}
\bibliography{biblio}

\end{document}